\documentclass[11pt,a4paper,oneside]{amsart}   
\usepackage[utf8]{inputenc}
\usepackage{amsmath,amsthm}     
\usepackage{lmodern}
\usepackage[T1]{fontenc}    
\usepackage[american]{babel}
\usepackage{chngcntr}
\usepackage{apptools}
\usepackage{doi}
\usepackage{graphicx}
\numberwithin{equation}{section}
\usepackage{url}
\usepackage{esint}
\usepackage{dsfont}
\usepackage{csquotes}
\newcommand{\ID}{\mathds{D}}
\newcommand{\IK}{\mathds{K}}
\usepackage[foot]{amsaddr}
\newtheorem{mydef}{Definition}
\usepackage[utf8]{inputenc}
\usepackage{wrapfig}
\newcommand{\abl}[2]{\frac{\partial #1}{\partial #2}}
\usepackage{amsfonts}
\usepackage{hyperref}
\usepackage[font=small]{caption} 
\usepackage{color} 
\usepackage{subcaption,graphicx}
\theoremstyle{definition}
\newtheorem{Def}{Definition}[section]
\theoremstyle{definition}
\newtheorem{Bem}[Def]{Remark}

\newtheorem{Korollar}[Def]{Corollary}

\newtheorem{Example}[Def]{Example}
\newtheorem{Lemma}[Def]{Lemma}
\theoremstyle{definition}
\newtheorem{Theorem}{Theorem}
\AtAppendix{\counterwithin{Theorem}{section}}
\usepackage[
            style=numeric-comp,
            sorting=nyt,
            abbreviate=true,
            useprefix=true,
            giveninits=true,
            hyperref=true,
            maxcitenames=3,
            maxbibnames=20,
            uniquename=init,
            sortcites=true,
            doi=true,
            isbn=true,
            url=false,
            eprint=false,
            backend=biber,
            date = year
                           ]{biblatex} 
  \renewbibmacro{in:}{}
    \DefineBibliographyStrings{english}{bibliography = {References}}
    \DeclareFieldFormat
  [article,inbook,incollection,inproceedings,patent,thesis,unpublished]
  {titlecase:title}{\MakeSentenceCase*{#1}}

\DeclareMathOperator*{\esssup}{ess\,sup}

\title[Local existence in micro-macro models]{Local existence of strong solutions to micro-macro models for reactive transport in evolving porous media}
\author{Stephan Gärttner$^1$}
\thanks{Correspondence: gaerttner@math.fau.de} 
\author{Peter Knabner$^{1,2}$}
\author{Nadja Ray$^2$}

\address[1]{Department of Mathematics, Friedrich-Alexander University Erlangen-Nuremberg, Cauerstr. 11, 91058 Erlangen, Germany}

\address[2]{Stuttgart Center for Simulation Science (SC SimTech), University of Stuttgart, Pfaffenwaldring 5a, 70569 Stuttgart, Germany}

\date{January 28, 2022}
\begin{document}

\maketitle
%\tableofcontents

\begin{abstract}
Two-scale models pose a~promising approach in simulating reactive flow and transport in evolving porous media. Classically, homogenized flow and transport equations are solved on the macroscopic scale, while effective parameters are obtained from auxiliary cell problems on possibly evolving reference geometries (micro-scale). Despite their perspective success in rendering lab/field-scale simulations computationally feasible, analytic results regarding the arising two-scale bilaterally coupled system often restrict to simplified models. In this paper, we first derive smooth-dependence results concerning the partial coupling from the underlying geometry to macroscopic quantities. Therefore, alterations of the representative fluid domain are described by smooth paths of diffeomorphisms. Exploiting the gained regularity of the effective space- and time-dependent macroscopic coefficients, we present local-in-time existence results for strong solutions to the partially coupled micro-macro system using fixed-point arguments. What is more, we extend our results to the bilaterally coupled diffusive transport model including a~level-set description of the evolving geometry.
\end{abstract}
\vspace{0.5cm}
\noindent\textbf{MSC classification: } {35A01, 35B30, 35M30, 35Q49}\\
\noindent\textbf{keywords: } \textit{evolving porous media, level-set equation, reactive flow \& transport,\\
local-in-time existence, strong solutions} 

\section{Introduction}
\label{SEC:Introduction}
Reactive transport in evolving porous media gained increasing interest over the last dec\-ades due to the wide range of applicability from enhanced oil recovery to biofilm growth~\cite{SeigneurReview2019}. Multi\-scale models pose a~powerful scheme to capture not only the flow and solute transport within the fluid, but also the evolution of the porous medium's properties due to structural alteration, e.g.\ triggered by agglomeration/precipitation or dissolution processes.  A~derivation of such models is provided by periodic homogenization of pore-scale models which are in turn based on first principles. For a general introduction to multi-scale approaches in reactive transport modelling, we refer to~\cite{MolinsKnabnerReview2019}.

In two-scale models for reactive transport in evolving porous media such as derived in~\cite{vanNoorden09,PhaseFieldModelPop}, flow and reactive transport equations are typically solved on the macroscopic domain. These PDEs encompass several effective parameters as coefficients such as porosity or diffusivity that are connected to the underlying microscopic geometry. Due to the evolution of the porous medium, the arising parameters depend on both space and time. As this evolution is for instance often driven by chemical reactions, i.e.\ dependent on the solution of the transport equation, the type of models considered here inherently features a~two-way coupling between the scales complicating analytical treatment. We illustrate the coupling of the macroscopic equations to the underlying geometry resolved in representative unit cells in Figure~\ref{fig:Coupling}. Henceforth, we distinguish two different types of coupling. The two-way coupling between both scales will be referred to as \textit{full coupling}. Commonly, also a~simplified coupling structure is investigated in the literature, disregarding the back-coupling from the macro to the micro-scale, cf.~Figure~\ref{fig:Coupling}. We refer to the arising one-sided coupling as a~\textit{partial coupling} scenario.

\begin{figure}[!h]
    \centering
    \includegraphics[width=0.8\textwidth]{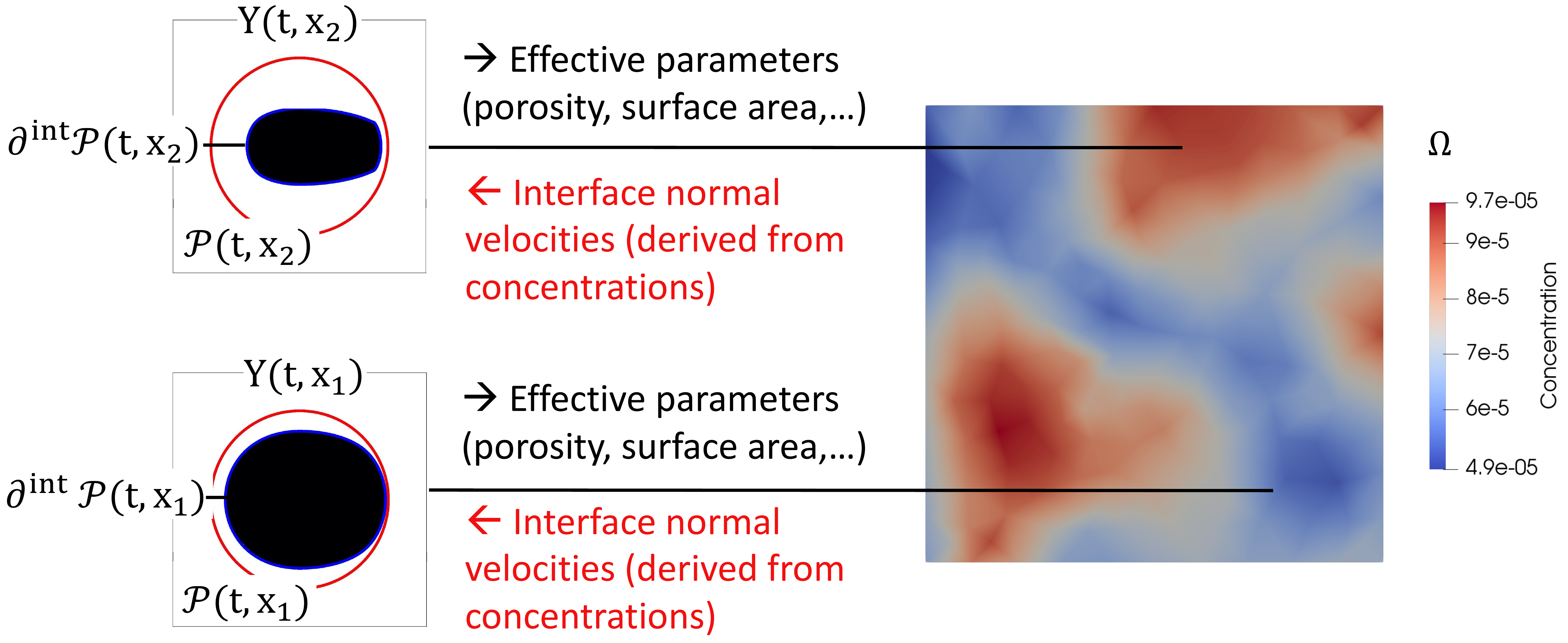}
    \caption{Schematic presentation of micro-macro coupling in multi-scale reactive transport models. Geometry dependent effective parameters influence the macroscopic flow and solute transport, see black-colored arrows. In the fully coupled scenario, solute concentrations on the macroscopic domain $\Omega$ additionally prescribe the evolution of the underlying microscopic geometry, see red-colored arrows. }
    \label{fig:Coupling}
\end{figure}

An additional challenge is posed by a~suitable framework to capture the evolving geometries. A variety of convenient methods is outlined in~\cite{EvolvingMicroAlber} regarding the description of evolving micro-structures. Commonly, level-set methods or phase-field approaches are used, in which case the macroscopic concentrations either prescribe a~normal interface velocity, cf.~\cite{vanNoorden09}, or induce a~source term to the phase field~\cite{PhaseFieldModelPop}, respectively. However, especially for highly symmetrical shapes, geometry evolution is often modelled and simulated in a~simplified way via ODEs for typical characteristic parameters such as the porosity in case of~\cite{Schulz2017,SchulzBiofilm} or the thickness of the precipitation layer as in~\cite{vanNoorden_strip}. Furthermore, models can also handle the geometry evolution implicitly by computing porosity from conservation of mass and deducing all geometry-related effective quantities therefrom by heuristic laws, cf.~\cite{Jambhekar16,Soulaine2016}, facilitating the numerical treatment. In \cite{PaperGaerttner2020,bastidas2020adaptive}, micro-macro models including fully resolved microscopic geometries have been investigated numerically. Regarding the most general setting, the associated evolution PDE (level-set or phase-field equation) is solved on reference geometries virtually attached to each degree of freedom of the macroscopic discretization.

In the literature, different approaches are present to obtain existence results for the effective reactive transport model including geometry alterations. In~\cite{PETER09, Gahn2021}, perforated microscopic domains are mapped onto a~periodic reference domain using diffeomorphisms. Existence results for the transformed microscopic equations on the reference domain are then leveraged to the effective model by means of two-scale convergence. Typically, the upscaling process of the transformed model is of increased complexity due to the appearance of additional factors in the highest-order terms arising from the transformation itself. Restricting to diffusion-reaction systems, \cite{PETER09} presents existence results to the fully coupled system with geometry evolution modelled via an~ODE for the determinant of the deformation gradient. As such, only effects emerging from changes in the pore-space volume are reflected. Likewise, diffusion-advection-reaction equations are treated in~\cite{Gahn2021} assuming an~a-priori given geometry evolution (partial coupling).

On the other hand, existence results can also be derived by investigating the effective model itself.
Existence of weak solutions to a~homogenized diffusion-driven model with partial and full coupling between the scales was shown in~\cite{PhdMeier} using transformations of the model equations onto fixed reference domains. However, the analysis performed requires the smoothness of the effective diffusion tensor as an~additional assumption (partial coupling) or neglects its evolution with time completely (full coupling). Likewise, in~\cite{Schulz2017}, fully coupled systems with diffusion-driven transport are investigated. Furthermore,~\cite{Schulz2020} considered partially coupled systems with advection including degenerating hydrodynamic parameters. Yet, a~key assumption is the a-priori knowledge of the relation between effective parameters and the porosity which plays the role of an~order parameter. 

In this paper, we follow the latter approach performing analysis directly on the effective micro-macro model. As a~key result, we prove smooth dependence of effective parameters on the underlying geometry. More precisely, we consider smoothly bounded underlying microscopic geometries and restrict to setups precluding degeneracy of effective parameters. By describing geometry alterations by smoothly parameterized paths of diffeomorphisms, we make use of an~additional parameter characterizing its state. As such, the presented framework covers a~broad range of geometry evolution which does not rely on parameterizability by a~single physical quantity such as the porosity. On the one hand, we use these results to investigate the existence of strong local-in-time solutions to the partially coupled model potentially including advective solute transport. On the other hand, we treat the scenario of full coupling and diffusion-driven transport. In that case, the macroscopic concentrations are coupled to the full level-set equation for geometry evolution. Thereby, we extend similar results known in the literature for simple and restrictive geometries like cubes or spheres, cf.~\cite{Schulz2017}. 

Our paper is outlined as follows: In Section~\ref{SEC:Model}, we present an~established model for reactive flow and transport in evolving porous media. Restricting to solute transport by diffusion only in Section~\ref{SEC:DiffTransport}, we derive smooth dependence results for the diffusion tensor on the geometry and prove local-in-time existence to the partially and fully coupled model. Establishing analogous results for the permeability tensors in Section~\ref{SEC:AdvTransport}, smooth dependence of the Darcy velocity field on the underlying geometry is shown. Finally, existence of solutions to the partially coupled model including advective solute transport is proven in Section~\ref{SEC:AdvectExistence}.  

\section{Model}
\label{SEC:Model}

This research is based on a~micro-macro model for reactive flow and transport in evolving porous media introduced by~\cite{vanNoorden09} where its derivation from a~detailed pore scale model by formal homogenization arguments was performed in two spatial dimensions. A~generalization to three dimension was derived in~\cite{Schulz2017} by modification of the original deduction. More precisely, the model under consideration consists of a~transport equation for a~solute chemical species $c$ on the macroscopic domain of interest $\Omega \subset \mathbb{R}^d$ with $d\in \{2,3\}$:
\begin{align}
\label{EQ:transport}
    \partial_t (\phi c) + \nabla \cdot({v} c) - \nabla\cdot(\mathbb{D}\nabla c) =   \sigma f(c)  \quad \text{in } (0,T)\times \Omega,
\end{align}
where $\phi$ denotes the porosity, $\sigma$ the specific surface area, $v$ the flow velocity, $\mathbb{D}$ the effective diffusion tensor and $f$ a~source/sink term due to heterogeneous reactions. As such, the parabolic equation~(\ref{EQ:transport}) models solute transport by diffusion and advection processes as well as chemical reactions at the fluid-solid interface.

The effective parameters in~(\ref{EQ:transport}) are derived from unit-cells $Y= \left(-\frac{1}{2}, \frac{1}{2} \right)^d$ representing a~local reference elementary volume of the underlying geometry. The respective exterior boundary is denoted by $\partial Y$. For the following, we consider solid inclusions compactly contained within $Y$. Let us denote the remaining fluid domain therein by $\mathcal{P}$ and the interior boundary by $\partial  ^\text{int}\mathcal{P}$, cf.~Figure~\ref{fig:Coupling}. 
The diffusion tensor is then given as $\mathbb{D}_{i,j} := \int_{\mathcal{P}}\left(\partial_{y_i}\zeta_j + \delta_{ij}\right)\,dy$ for $i,j\in\lbrace 1,\dots,d\rbrace$ with Kronecker delta $\delta_{ij}$, where $\zeta_j$ are the solutions to the following elliptic problems:
\begin{align}
\label{Def: DiffusionTensor}
- \nabla_y\cdot(\nabla_y \zeta_j) &= 0 && \text{in } {\mathcal{P}}, \nonumber \\
\nabla_y \zeta_j\cdot\nu &=  - e_j\cdot\nu && \text{on } {\partial^\text{int}\mathcal{P}},\\
\zeta_j \text{ periodic in } y, &\quad \int\limits_{\mathcal{P}} \zeta_j \; dy = 0,  \nonumber
\end{align}
with outer unit normal $\nu$. For the derivation in the context of evolving geometries, see~\cite{vanNoorden09}. In this case, the evolution is reflected in time-dependent domains $\mathcal{P}=\mathcal{P}(t)$. Note that the shape of the elliptic problem~(\ref{Def: DiffusionTensor}) is identical to the one derived under the assumption of fixed underlying geometries in~\cite{bookHornung}, leading to fixed domains $\mathcal{P}$.
We refer to~(\ref{Def: DiffusionTensor}) as diffusion cell-problems. 

The advective flow field $v$ and the associated pressure field $p$ are given by Darcy's equation:
\begin{align}
\label{EQ:Darcy}
   v &= -\frac{\mathbb{K}}{\mu} \nabla p  &&\text{in } \Omega,\; t\in(0,T) , \\
\nabla \cdot v &=0 && \text{in } \Omega,\; t\in(0,T),  \nonumber 
\end{align}
with viscosity $\mu$ and permeability tensor $\mathbb{K}$. As $\mu$ is a constant being characteristic to the solvent, we set it to one for convenience. Note that the condition of a~divergence-free velocity field in~(\ref{EQ:Darcy}) is a~common simplification as discussed in~\cite{Gaerttner2020bPreprint}.  Due to the much larger time-scale of geometry evolution compared to fluid flow, it is justified to disregard the flow induced by fluid displacement arising from a~variable pore-space volume. 

The permeability tensor in~(\ref{EQ:Darcy}) is defined as
$\mathbb{K}_{i,j} := \int_\mathcal{P} \omega_j^i \; dy$ for $i,j\in\lbrace 1,\dots,d\rbrace$, where $(\omega_j, \pi_j)$ are the solutions to the Stokes-type problems,  cf.~\cite{vanNoorden09}:
\begin{align}
\label{Def: PermTensor}
-\Delta_y \omega_j + \nabla_y \pi_j &= e_j && \text{in } {\mathcal{P}} ,\nonumber \\
\nabla_y \cdot \omega_j  &=0 && \text{in } {\mathcal{P}} ,\\
\omega_j &=0 &&\text{on } {\partial^\text{int}\mathcal{P}}, \nonumber \\
\omega_j , \; \pi_j \text{ periodic in } y, &\quad \int\limits_{\mathcal{P}}\pi_j \; dy = 0.  \nonumber
\end{align}
Likewise, the shape of the Stokes-type problem is identical to the one derived under the assumption of fixed underlying geometries in \cite{bookHornung}. In case of an~evolving geometry, we again obtain $\mathcal{P}=\mathcal{P}(t)$. We refer to~(\ref{Def: PermTensor}) as permeability cell-problems. Note that both cell-problems (\ref{Def: DiffusionTensor}), (\ref{Def: PermTensor}) result in symmetric positive semi-definite tensors $\ID$, $\IK$.

 Finally, we assume the existence of a~level-set function $\Phi_0:\Omega\times Y \to \mathbb{R}$ characterizing the solid part within the unit-cell attached to the macroscopic point $x\in\Omega$ at initial time by $\{\Phi_0(x,\cdot)>0\}$. Consequently,  $\{\Phi_0(x,\cdot)<0\}$ refers to the fluid domain~$\mathcal{P}$ and $\{\Phi_0(x,\cdot)=0\}$ denotes the fluid-solid interface $\partial^\text{int}\mathcal{P}$. We require that the gradient of $\Phi_0$ does not vanish along the zero-level-set to ensure the representation of a~submanifold of codimension one. For a~normal interface velocity field $v_n:(0, T)\times \Omega \times Y \to \mathbb{R}$ the evolution of $\Phi_0$ is described by the level-set equation for $\Phi:(0, T)\times \Omega \times Y \to \mathbb{R}$, cf.~\cite{Sethian99}:
 \begin{align}
\label{EQ:levelSet}
\frac{\partial \Phi}{\partial t} + v_n |\nabla_y \Phi| &= 0 &&\quad \text{in }  (0, T)\times \Omega \times Y,  \\
\Phi(0,\cdot,\cdot) &= \Phi_0  &&\quad \text{in } \Omega \times Y.\nonumber
\end{align}
The different sub-domains of $Y$ (fluid-domain, solid-domain, interface) at a certain time $t$ are encoded by the sign of $\Phi(t,\cdot)$ according to the convention for the initial condition $\Phi_0$ above.
Eventually, the interface velocity is coupled to the chemical reaction in a~mass conserving way, e.g.
\begin{align}
\label{DEF:NormalVel}
    v_n(t,x,y) =- v_{\text{mod}}(y) f(c(t,x)), 
\end{align}
potentially using a~scalar speed modification function $v_{\text{mod}}$ which allows for a~varying normal interface velocity within a~unit-cell, cf.~\cite{PaperGaerttner2020}.

\section{Smooth parameter dependence and existence for diffusive transport}
\label{SEC:DiffTransport}
In this section, we consider a~special case of the model introduced in Section~\ref{SEC:Model}. Neglecting advective transport for the solute species, we focus on the coupling from the micro to the macro-scale conveyed by $\phi,\sigma$ and $\ID$ only, facilitating the analysis. After discussing the setup in more detail, this section first considers the smoothness of the partial coupling. Therefore, we investigate the dependence of the diffusion tensor $\ID $ on deformations of the microscopic geometry via diffeomorphisms in Section~\ref{SEC:diff} as the principle step to establish existence results for the partially coupled problem in Section~\ref{SEC:DiffExist1Sided}. Moreover, we show the induction of suitable diffeomorphisms by the level-set equation in Section~\ref{SEC:FlowInducedDiffeomorphisms}, ultimately leading to local-in-time existence results for the fully coupled problem in Section~\ref{SEC:DiffExist2Sided}.

\subsection{Setting}
\label{SEC: DiffSetting}
 For the following, we consider solute transport by diffusion only. As the term $\partial_t (\phi c)$ in (\ref{EQ:transport}) is difficult to handle analytically, it is shifted to the right-hand side, cf.~\cite{Schulz2017}. Therefore, we write  
\begin{align}
\label{EQ:transportDIFF}
    \phi \partial_t c - \nabla\cdot(\mathbb{D}\nabla c) =   \sigma f(c) - \partial_t \phi c \quad \text{in } (0,T)\times \Omega,
\end{align}
rendering the equations for flow and permeability determination~(\ref{EQ:Darcy}),~(\ref{Def: PermTensor}) superfluous. Hence, it is sufficient to close the model regarded in this section by~(\ref{Def: DiffusionTensor}), (\ref{EQ:levelSet}) and~(\ref{DEF:NormalVel}). We emphasise that the following considerations assume good-natured conditions such as the smoothly bounded solid geometry being compactly contained within the unit-cell~$Y$. By considering local-in-time estimates, this setting is maintained by an~appropriate choice of initial conditions. Using diffeomorphisms to describe solid alteration, we particularly exclude clogging scenarios and degenerating equations. 

\subsection{Continuous dependence of diffusion tensors}
\label{SEC:diff}
In order to prove existence to the model described in Section~\ref{SEC: DiffSetting}, we make extensive use of existence theory for linear parabolic equations, cf. Theorem~\ref{TheoremLadyzenkaya} in the appendix. As such, we require moderate regularity for $\ID$ as the coefficient of the leading order term in~(\ref{EQ:transportDIFF}).

Therefore, this section is concerned with the dependence of the diffusion tensor $\mathbb{D}$ on the evolving geometry. The method presented consists of three steps: At first, we establish higher regularity for weak solutions to the diffusion cell-problem~(\ref{Def: DiffusionTensor}). Although using standard methods, we state the results in detail due to the uncommon periodic boundary conditions. Based on that, a~mapping between the geometry and the elliptic PDE's solution of desired regularity is constructed using the implicit function theorem following the technique of~\cite{Henry}. Finally, we extend our smoothness results from the solutions $\zeta_j$, $j\in \{1,\dots,d\}$, of~(\ref{Def: DiffusionTensor}) to the diffusion tensor $\ID$ which is given as an~affine-linear functional of $\zeta_j$.

For the formulation of problem~(\ref{Def: DiffusionTensor}) and our regularity result Lemma~\ref{Lemma: H2regular}, we assume $\partial^\text{int}\mathcal{P}$ to be $C^{2,1}$-regular. As it becomes apparent in Theorem~\ref{THEOREM1}, it is necessary to consider the full class of elliptic PDEs of type~(\ref{Def: DiffusionTensor}) with general source term and Neumann boundary conditions.   
In a~first step towards a~suitable weak formulation of problem (\ref{Def: DiffusionTensor}), we introduce periodic Sobolev spaces according to~\cite{CioranescuIntro}. Let $\mathcal{P}_\text{ext} \subset \mathbb{R}^d$ denote the perforated domain obtained by periodic extension of $\mathcal{P}$ in $\mathbb{R}^d$.
Then we define $H_\#^k (\mathcal{P})$ for $k\in \mathbb{N}$ as the closure of Y-periodic functions in $C^\infty(\mathcal{P}_\text{ext})$ with respect to the $H^k$-norm. 
For the weak formulation of the diffusion cell-problem with Neumann boundary conditions, we introduce the following function spaces
\begin{align*}
    H^k_{\#,0}(\mathcal{P}) = \left\lbrace v\in H_\#^k ( \mathcal{P}):  \; \int_{\mathcal{P}} v \; dy =0 \right\rbrace,
\end{align*}
equipped with the norm $||u||_{ H^k_{\#,0}(\mathcal{P})} = ||u_{|\mathcal{P}}||_{ H^k(\mathcal{P})} $.

Accordingly, the weak problem with general source term $f$ and Neumann boundary condition $g$ reads: Find \mbox{$u\in  H^1_{\#,0}(\mathcal{P})$} such that
\begin{align}
\label{EQ: DiffWeakForm}
    \int_\mathcal{P} \nabla u \cdot \nabla v \; dy - \int_{\partial^\text{int}\mathcal{P}} g v \; d\sigma = \int_\mathcal{P} f v \; dy, \quad \forall v\in  H^1_{\#}(\mathcal{P}),
\end{align}
with $g\in L^2(\partial^\text{int}\mathcal{P})$, $f\in L^2(\mathcal{P})$. Apparently, the compatibility condition $\int\limits_\mathcal{P} f \; dx =- \int\limits_{\partial^\text{int}\mathcal{P}} g \; d\sigma$ is necessary for solvability. 

Next, we establish unique solvability for the above problem and present conditions ensuring solutions to be of higher regularity. Particularly, a~$C^{2,1}$-regular interior boundary $\partial^\text{int}\mathcal{P}$ proves to be sufficient for equations (\ref{Def: DiffusionTensor}) to hold in the $L^2$-sense and therefore point-wise almost everywhere in~$\mathcal{P}$. 

\begin{Lemma} \textit{Elliptic Regularity}\\
Problem (\ref{EQ: DiffWeakForm}) has a~unique solution $u\in H^1_{\#,0}(\mathcal{P})$. Let $k$ be an~integer number $k\geq 0$. If the interior boundary $\partial^\text{int}\mathcal{P}$ is furthermore  $C^{k+2,1}$-regular and  $g\in H^{k+\frac{1}{2}}(\partial^\text{int}\mathcal{P})$, $f\in H_\#^k(\mathcal{P})$ fulfilling the compatibility condition $\int\limits_\mathcal{P} f \; dx = -\int\limits_{\partial^\text{int}\mathcal{P}} g \; d\sigma$, then $u\in H^{k+2}_{\#,0}(\mathcal{P})$.
\label{Lemma: H2regular}
\end{Lemma}

\begin{proof}
 The coercivity of the bilinear form on the Hilbert-space $H^1_{\#,0}(\mathcal{P})$ is guaranteed by Poincaré's inequality. The Lax-Milgram theorem therefore implies the unique existence of a~solution $u \in H^1_{\#,0}(\mathcal{P})$ to the weak formulation, cf.~\cite{CioranescuIntro}.

Let us now suppose that $\partial^\text{int}\mathcal{P}$ is $C^{k+2,1}$-regular.
By the surjectivity of the higher-order trace operator continuously extending the following function, cf.~\cite{Necas},
\begin{align*}
    \text{Tr}_k &: H^{k+2}(\mathcal{P}) \to \prod_{l=0}^{k+1} H^{k-l+\frac{3}{2}} (\partial^\text{int}\mathcal{P}), \quad \forall u \in H^{k+2}(\mathcal{P})\cap C^{k+1}(\Bar{\mathcal{P}}), \\
    \text{Tr}_k u &= \left(u|_{\partial^\text{int}\mathcal{P}}, \partial_\nu u|_{\partial^\text{int}\mathcal{P}},\cdots, \partial_\nu^{k+1} u|_{\partial^\text{int}\mathcal{P}}  \right), \nonumber
\end{align*}
we find a~function $\Psi \in H^{k+2}_{\#,0}(\mathcal{P})$ with $\abl{\Psi}{\nu} = \abl{u}{\nu}$ on $\partial^\text{int}\mathcal{P}$ in the trace sense vanishing in a~neighborhood of $\partial Y$. Exploiting linearity of the problem and carrying out the subsequent argument for $u-\Psi$ and the corresponding source term $\tilde{f}=f-\Delta \Psi \in H^k(\mathcal{P})$, we can assume homogeneous Neumann boundary conditions. By Theorem~3 of~\cite{mikhailov}, the required interior higher regularity is established. Following the lines of Theorem~4 of~\cite{mikhailov}, higher regularity is also obtained in a~neighborhood of every interior boundary point $y\in\partial^\text{int}\mathcal{P}$. Using an~open covering argument, we obtain $u \in H^{k+2}_{\#,0}(\mathcal{P})$.
\end{proof}

\begin{Bem}
\label{REM:DiffTensorIndeedSolve}
As $\partial^\text{int}\mathcal{P}$ is of class $C^{2,1}$, $\nu \cdot e_1 \in H^{\frac{1}{2}}(\partial^\text{int}\mathcal{P})$ holds. Furthermore, we have $\int_{\partial^\text{int}\mathcal{P}} \nu \cdot e_1  \; d\sigma = 0$ by Gauss's theorem. As such, the last lemma covers the unique solvability of problem (\ref{Def: DiffusionTensor}) in $H^{2}_{\#,0}(\mathcal{P})$ and the equation holds in a~point-wise almost-everywhere sense. 
\end{Bem}

\begin{Bem}
Note that the higher regularity result of Lemma~\ref{Lemma: H2regular} holds for the diffusion cell-problems of the homogenized model. However, regarding the upscaling process from the associated pore-scale diffusion equation to the effective one considered here, the convergence of the sequence of transport problems $c_\varepsilon$ defined on the $\varepsilon$-periodic domains $\Omega_\varepsilon$ against the homogenized solution $c$ is not valid with respect to these stronger norms. This is essentially due to the absence of uniform boundedness of $c_\varepsilon$ with respect to $\varepsilon$. 
\end{Bem}

Remark~\ref{REM:DiffTensorIndeedSolve} enables us to define mappings from the solution space of equation~(\ref{EQ: DiffWeakForm}) to the bulk and boundary data spaces with a~point-wise interpretation.  More precisely, we consider such mappings which involve the geometry alteration as a~parameter. This idea poses the main ingredient in the subsequent investigation of the dependence of solutions to the weak problem (\ref{EQ: DiffWeakForm}) and functionals thereof on the geometry $\mathcal{P}$. Following the approach presented by \cite{Henry}, a~Lagrangean description of the initial problem on varying domains is taken. The main step is to rewrite the equation on a~fixed domain of reference $\mathcal{P}$. This technique has also been successfully applied to the homogenization of PDEs on non-uniformly periodic or evolving domains in the context of porous media \cite{HornungProceed,PETER07,Gahn2021} or shape optimization minimizing energy functionals depending on a~PDE's solution~\cite{Simon}. In order to re-define functions mapping from the altered domains as functions on a~fixed domain of reference, we make use of the concept of diffeomorphisms and pullbacks:  

\begin{mydef} \textit{Diffeomorphism}\\
Let $h:Y \to h(Y)\subset\mathbb{R}^d$ be a~bijective mapping with $h\in C^{k,\alpha}(Y, \mathbb{R}^d)$ for $k\geq 1,$ $\alpha\in [0,1]$. We call $h$ a~diffeomorphism of class $\text{Diff}^{k,\alpha}(Y, \mathbb{R}^d)$ iff the inverse $h^{-1}$ satisfies \mbox{$h^{-1}\in C^{k,\alpha}(h(Y),Y)$}.
\label{DEF: diffeo}
\end{mydef}
The action of a~diffeomorphism on the set $\mathcal{P}\subset Y$ is illustrated in Figure~\ref{fig:Deformation}.
\begin{mydef} \textit{Pullback}\\
Let $h$ be a~diffeomorphism $h\in \text{Diff}^{1}(Y)\subset C^1(Y, \mathbb{R}^d)$ and $l:h(Y) \to \mathbb{R}^n$, $n\in\mathbb{N}$. We define the pullback $h^*$ by 
\begin{align*}
h^*(l):Y \to \mathbb{R}^n, \quad h^* l(x) := l(h(x)).
\end{align*}
\label{DEF: pullback}
\end{mydef} 

\begin{figure}[!h]
    \centering
    \includegraphics[width=0.6\textwidth]{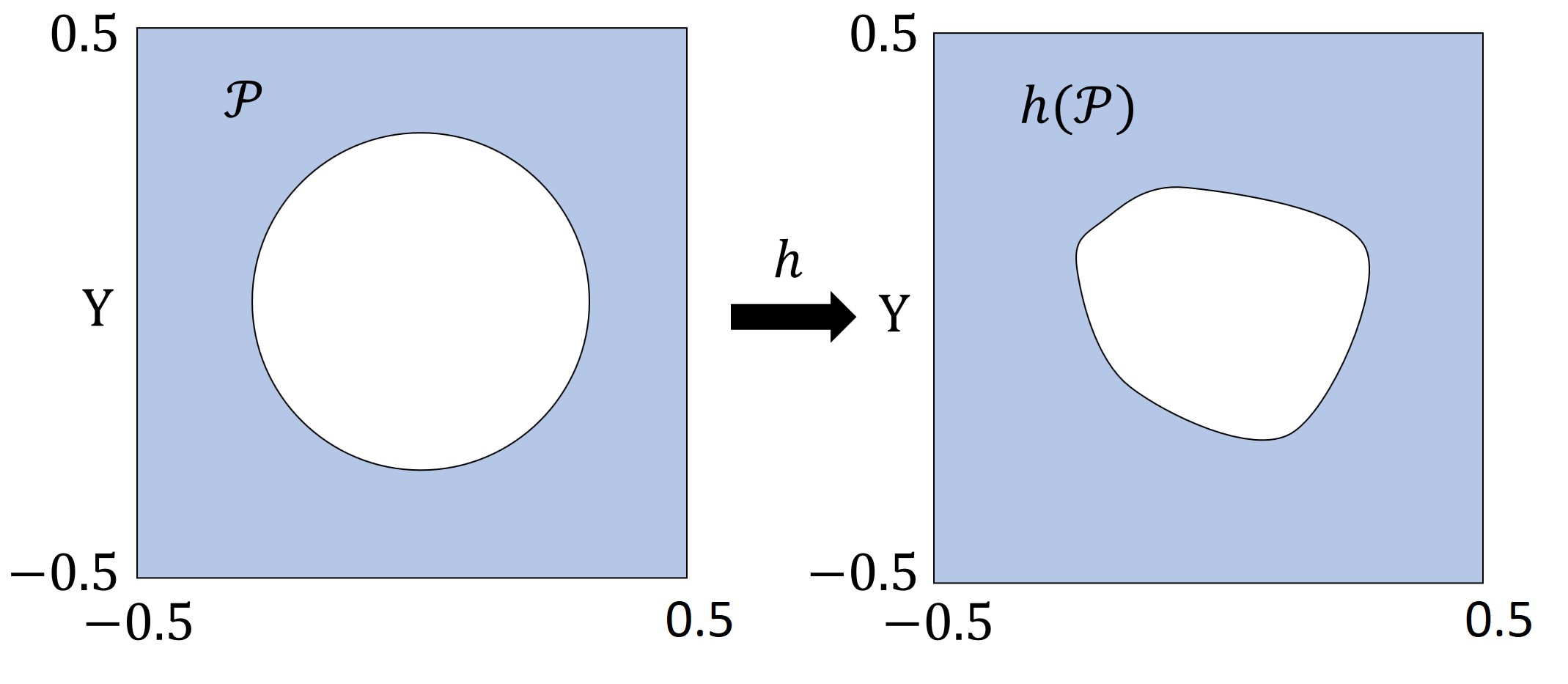}
    \caption{Smooth deformation of domain $\mathcal{P}$ with circular inclusion mediated by a~diffeomorphism $h\in \text{Diff}_\Box^{1} (\bar{Y})$. As $h$ preserves the exterior boundary $\partial Y$, the image-set $h(\mathcal{P})$ is an~admissible unit-cell pore-space geometry.}
    \label{fig:Deformation}
\end{figure}
 Note that the pullback is a~linear operation and its inverse is given as $(h^*)^{-1}(l) = (h^{-1})^*(l)$. Furthermore, for sufficiently smooth diffeomorphisms, the pullback is a~bounded operator between the $H^m$ function spaces on the original and deformed set, cf.~\cite{mcleanElliptic,Henry}. In Example~\ref{Ex:DiffeoCircle}, we present an~explicit construction of a~family of diffeomorphisms mapping circular inclusion of different radii to one another and provide illustrations of the associated pullback of a~distance function in Figure~\ref{fig:DiffeoCircles}.   

Using the tools of Definition~\ref{DEF: diffeo} and~\ref{DEF: pullback}, we are now able to characterize solutions to~(\ref{EQ: DiffWeakForm}) on the domain $h(\mathcal{P})$ as roots of a~function~$F$. Taking a~Lagrangean point of view, we work on a fixed domain of reference $\mathcal{P}$. As such, functions first need to be conveyed to the deformed domain $h(\mathcal{P})$ where the differential operators according to~(\ref{Def: DiffusionTensor}) are applied. Performing a~pullback with the inverse deformation, functions are translated back onto the domain of reference.
Accordingly, let us consider the following mapping:
\begin{align}
    \label{EQ:F}
    F & :  H^2_{\#,0}(\mathcal{P}) \times \text{Diff}^{2,1}(\bar{Y}) \to L_0^2(\mathcal{P}) \times H^\frac{1}{2}(\partial^\text{int}\mathcal{P}),\\
    F & (u,h) = \left(h^*\Delta h^{*-1} u - \fint\limits_\mathcal{P} h^*\Delta h^{*-1} u \; dy,\; h^* \text{Tr}_{h(\mathcal{P})}\left(\nu_{h(\mathcal{P})}  \cdot (\nabla h^{*-1} u - e_1)  \right) \right) = (F_1,F_2), \nonumber
\end{align}
where $\nu_{h(\mathcal{P})}$ denotes the outer unit normal with respect to the domain $h(\mathcal{P})$ and $\text{Tr}_{h(\mathcal{P})}$ the standard trace operator on $h(\mathcal{P})$. Note that the normal vectors can be extended within a~tubular neighborhood of $\partial^\text{int}\mathcal{P}$, cf.~Theorem~\ref{TheoremTubNeigh}. By the trace theorem and change of variable rule, $F$ is well-defined as a~mapping between the stated spaces. Note that in the first component~$F_1$ a~vanishing mean value is enforced. Thereby, we eliminate an~additional degree of freedom to ensure surjectivity of~$F$.

Since the image of the unit-cell under an~arbitrary diffeomorphism may not be a~unit-cell, we must also introduce a~restricted class of deformations. Therefore, let $\text{Diff}_\Box(\bar{Y})$ denote the set of diffeomorphism preserving the exterior boundary $\partial Y$ defined by:
\begin{align}
    \label{EQ:DiffBoundaryPreserving}
    \text{Diff}_\Box^{k,\alpha} (\bar{Y}) = \left\lbrace h \in \text{Diff}^{k,\alpha} (\bar{Y},\bar{Y}):\; h_{|U} = \text{id}_U  \right\rbrace,
\end{align}
for some fixed neighborhood $\partial Y \subset U\subset \bar{Y}$. In fact, the diffeomorphism illustrated in Figure~\ref{fig:Deformation} belongs to this specified class. The construction~(\ref{EQ:DiffBoundaryPreserving}) is inspired by the Hanzawa transformation, cf.~\cite{PrussSimonett}, and requires the diffeomorphism to decay smoothly towards the identity at the exterior boundary. As such, periodic functions with respect to $Y$ admit a~periodic pullback for $h\in  \text{Diff}^{k,\alpha}_\Box (\bar{Y})$. 
The importance of the mapping~$F$ defined in~(\ref{EQ:F}) is now given by the following characterization property:
\begin{align*}
     F(u,h) = (0,0) \iff \zeta_1 = h^{*-1} u - \fint\limits_{h(\mathcal{P})}h^{*-1} u\; dy \text{ solves (\ref{Def: DiffusionTensor}) weakly in } h(\mathcal{P}),
\end{align*}
for all $h\in \text{Diff}^{2,1}_\Box (\bar{Y})$. 

We will now apply the implicit function theorem to $F$ to obtain a~continuous mapping $h \mapsto u$ in the respective spaces as summarized in Theorem~\ref{THEOREM1} below. As a~result, a~relation between the deformation of $\mathcal{P}$ and the associated solution to (\ref{Def: DiffusionTensor}) is established. Therefore, we check the following properties: \\

\underline{Continuity of F:}
Let us denote the pullback of $u$ via $h^{-1}$ by $v=h^{*-1}u$ constituting a function on $h(\bar{Y})$ and the image point on $h(\bar{Y})$ by $y=h(x)$. First, we consider the first term of the first component $F_1$ of the mapping $F$. By applying the chain rule, we have:
\begin{align}
\label{EQ: ContinuityF1}
& h^*\Delta h^{*-1} u (x) = \Delta v(y) = \Delta u(h^{-1}(y)) = \\
&\sum_{i,j=1}^d \frac{\partial^2 u}{\partial x_i \partial x_j} \left(\sum_{k=1}^d \abl{h^{-1}_i}{y_k} \abl{h^{-1}_j}{y_k} \right)  +\sum_{i=1}^d \abl{u}{x_i} \left(\sum_{k=1}^d  \frac{\partial^2 h^{-1}_i}{\partial y_k \partial y_k} \right). \nonumber
\end{align}
Note that by the inverse function theorem and the representation of a~matrix' inverse via the cofactor matrix, we can rewrite all partial derivatives of $h_i^{-1}(y)$ as a~function of derivatives of $h_i(x)$ only involving multiplications and division by $\text{det}(\nabla h(x))$. By the uniform boundedness of the last expression away from zero, this map is in particular locally Lipschitz continuous.
For any sequence $(u_i,h_i)_{i \in \mathbb{N}}$ in $ H^2_{\#,0}(\mathcal{P}) \times \text{Diff}^{2,1}(\bar{Y}) $ converging to $(u,h)$ in the product topology, we have the convergence of the derivatives of $u_i$ in $L^2$ and the convergence of the derivatives of $h_i$ uniformly. Hence, expression (\ref{EQ: ContinuityF1}) is continuous in $(u,h)\in  H^2_{\#,0}(\mathcal{P}) \times \text{Diff}^{2,1}(\bar{Y})$. As the integral is a~linear and bounded operator \mbox{$\int : L^2(\mathcal{P})\to \mathbb{R}$}, we obtain continuity for $F_1$. 

We can apply the same strategy to prove continuity of the second component $F_2$ of $F$. Noting the representation
\begin{align*}
    \nu_{h(\mathcal{P})} (y) = (\nabla h)^{-T}(x) \; \nu_{\mathcal{P}}(x)  ||(\nabla h)^{-T}(x) \; \nu_{\mathcal{P}}(x)||^{-1}_2
\end{align*}
derived in \cite{Henry} using the inverse transposed Jacobian matrix $(\nabla h)^{-T}$, we compute
\begin{align*}
    F_{2}(u,h) (x) &=  h^{*} \text{Tr}_{h(\mathcal{P})}\left((\nu_{h(\mathcal{P})}  \cdot \nabla h^{*-1}u) (y) -  \nu_{h(\mathcal{P})}(y) \cdot e_1\right) \nonumber\\
    &= \text{Tr}_{\mathcal{P}}  \left(\left( (\nabla h)^{-T}  \nu_{\mathcal{P}} \right) \cdot \left[\sum_{i=1}^d\abl{u}{x_i} \abl{h^{-1}_i}{y_j}-\delta_{j,1}\right]_j \cdot ||(\nabla h)^{-T}(x) \; \nu_{\mathcal{P}}(x)||_2^{-1}\right) 
\end{align*} 
for $x\in \partial^\text{int}\mathcal{P}$.
As $\text{Diff}^{2,1}(\bar{Y}) \subset C^{2,1}(\bar{Y},\mathbb{R}^d)$ is open, we conclude the continuity of $F$ on a~neighborhood $V$ of $(u,\text{id}_{\bar{Y}})$.\\

\underline{Continuity of $F^{\prime}_u$:}
By Definition~\ref{DEF: pullback} the pullback operator is linear. Using the linearity of the differential and trace operators involved, we conclude
\begin{align*}
     F^{\prime}_u (u,h) (w) = \left(h^*\Delta h^{*-1} w -\fint_\mathcal{P} h^*\Delta h^{*-1} w \; dy, \; h^* \text{Tr}_{h(\mathcal{P})}(\nu_{h(\mathcal{P})}  \cdot \nabla h^{*-1} w) \right)
\end{align*}
for all $w\in H^2_{\#,0}(\mathcal{P})$. Following the arguments from above, we obtain continuous Fréchet differentiability with respect to the first argument on a~neighborhood $V$ of $(u,\text{id}_{\bar{Y}})$. \\

\underline{Bijectivity of $F^{\prime}_u$:}
The bijectivity of $F^{\prime}_u(u,\text{id}_{\bar{Y}})$ onto $L_0^2(\mathcal{P}) \times H^\frac{1}{2}(\partial^\text{int}\mathcal{P})$ is equivalent to finding a~unique solution to the elliptic problem~(\ref{Def: DiffusionTensor}) on $\mathcal{P}$. More precisely, for a~given point $(f,g)$ in the image space of $F$, we search for a~weak solution $u$ for the Neumann boundary condition $g$ and source term $f+c$ for a~constant $c\in\mathbb{R}$. By Lemma \ref{Lemma: H2regular}, there exists exactly one $c$ such that the problem admits a~solution (compatibility condition) in $H^2_{\#,0}(\mathcal{P})$. In that case, the solution is unique. Note that due to the  set of invertible bounded linear operators between Banach spaces being open and the continuity of $F^{\prime}_u$, the bijectivity property of $F^{\prime}_u$ in fact holds on a~neighborhood of $(u,\text{id}_{\bar{Y}})$.

Summarizing the above arguments, we conclude the following statement.

\begin{Theorem} \textit{Continuous dependence of $\ID$}\\
Assume a~$C^{2,1}$ open set $Y\setminus\bar{\mathcal{P}} \subset Y$ being compactly contained in $Y$ and $u\in H^2_{\#,0}(\mathcal{P})$ such that \mbox{$F(u,\text{id}_{\bar{Y}})= (0,0)$}, i.e.\ $u$ is a~solution to problem (\ref{EQ: DiffWeakForm}). Then there exists a~neighborhood \mbox{$V\subset C^{2,1}(\bar{Y})$} of $\text{id}_{\bar{Y}}$ and a~continuous function $g:V\to H^2_{\#,0}(\mathcal{P})$, $g(\text{id}_{\bar{Y}})=u$, such that
\begin{align*}
 F(g(h),h)= (0,0), \quad \forall h\in V.    
\end{align*}
Particularly, $h^{*-1}g(h)$ solves (\ref{EQ: DiffWeakForm}) up to an~additive constant on $h(\mathcal{P})$ for all \mbox{$h\in \text{Diff}^{2,1}_\Box (\bar{Y})\cap V$.}
\label{THEOREM1}
\end{Theorem}

\begin{proof}
This is an~immediate consequence of the implicit function theorem for Banach spaces as given in \cite{Valent} and the arguments above.
\end{proof}

By the previous theorem, we established a~continuous relation between the diffeomorphism~$h$ describing the alteration of the domain $\mathcal{P}$ and the pullback of the solution $g(h)$ to the associated problem~(\ref{Def: DiffusionTensor}). In a~final step, we show that the continuous dependence carries over to the desired quantity $\mathbb{D}$ defined in~(\ref{Def: DiffusionTensor}):
\begin{Korollar}
\label{Kor: DiffConti}
Under the assumptions of Theorem \ref{THEOREM1} the mapping 
\begin{align*}
    R: V \to \mathbb{R} ,\quad  h\mapsto \int_{h(\mathcal{P})}  \nabla h^{*-1}g(h) \; dy
\end{align*} 
is continuous. Therefore, the diffusion tensor depends continuously on $\text{Diff}^{2,1}_\Box (\bar{Y})$-variations of the domain $\mathcal{P}$.
\end{Korollar}
\begin{proof}
By the change of variables theorem and chain rule we have
\begin{align}
\label{EQ:KorDiffConti}
    \int_{h(\mathcal{P})} \nabla h^{*-1}g(h) \; dy &=
    \int_\mathcal{P} \nabla g_h(x) \cdot \nabla h^{-1}(h(x)) \cdot \mid \text{det} (\nabla h(x))\mid \; dx \\
    &=\int_\mathcal{P} \nabla g_h(x) \cdot \nabla h(x)^{-1} \cdot \mid \text{det} (\nabla h(x))\mid \; dx. \nonumber
\end{align}
As $g$ is continuous with respect to the $H^2$-norm in the image space, continuity of the functional is proven. 
\end{proof}

By the previous corollary, we established the continuous behavior of $\ID$ on $h\in\text{Diff}_\Box^{2,1}(\bar{Y})$ in the topology of $C^{2,1}(\bar{Y})$. However, this degree of regularity is insufficient for our later purposes, cf.~Theorem~\ref{TheoremSchulz}. In order to obtain stronger results, we specify the setting more tailored to our later application.  
Let us now consider a~smooth mapping \mbox{$h:(-S,S) \to \text{Diff}_\Box^{2,1}(\bar{Y})$} for some artificial time horizon $S>0$ and $h(0)=\text{id}_{\bar{Y}}$. This relates to a~1-parametric deformation of the initial geometry. 

\begin{Example}
\label{Ex:DiffeoCircle}
Consider inclusions $Y\setminus \bar{\mathcal{P}}$ of circular shape and of different radii. In this case, a~smooth diffeomorphism on $Y$ can be easily constructed by radially compressing/ex\-panding annuli within a~compact subset of $Y$. Given two radii $0<r_1\leq r_2<\frac{1}{2}$, a~deformation mapping a~circle of radius $r_2$ to a~circle of radius $r_1$ is defined by
\begin{align}
\label{EQ:DiffeoCircle}
    h_{r_1}(y)=\begin{cases}
    y, & |y|>\frac{1}{2}, \\
    (1-\xi( |y|))y + \xi (|y|)(\frac{r_1}{r_2} y), & r_2 \leq |y|\leq \frac{1}{2}, \\
    \frac{r_1}{r_2} y, & |y|\leq r_2
  \end{cases}
\end{align}
choosing a~suitable $\xi \in C^\infty \left([r_2,\frac{1}{2}]\right)$ with $\xi^\prime \in C_0^\infty \left((r_2,\frac{1}{2})\right)$, $\xi(r_2)=1$, $\xi(\frac{1}{2}) = 0$, see \cite{eden2021multiscale}. As such, the domain remains unchanged for $|y|>\frac{1}{2}$ and is uniformly contracted for $|y|\leq r_2$ with a~smooth convex-combination layer in between. For $r_2$ fixed, we can consider the path
\begin{align*}
    h\in C^0\left(0,r_2;\text{Diff}^{2,1}_\Box (\bar{Y})\right), \quad h:\; s\mapsto h_s.
\end{align*}
Then $R\circ h $ with $R$ being defined in Corollary \ref{Kor: DiffConti} is also a~continuous mapping. Consequently, the diffusion tensor depends continuously on $s$. 

\begin{figure}[!h]
    \centering
    \includegraphics[width=0.7\textwidth]{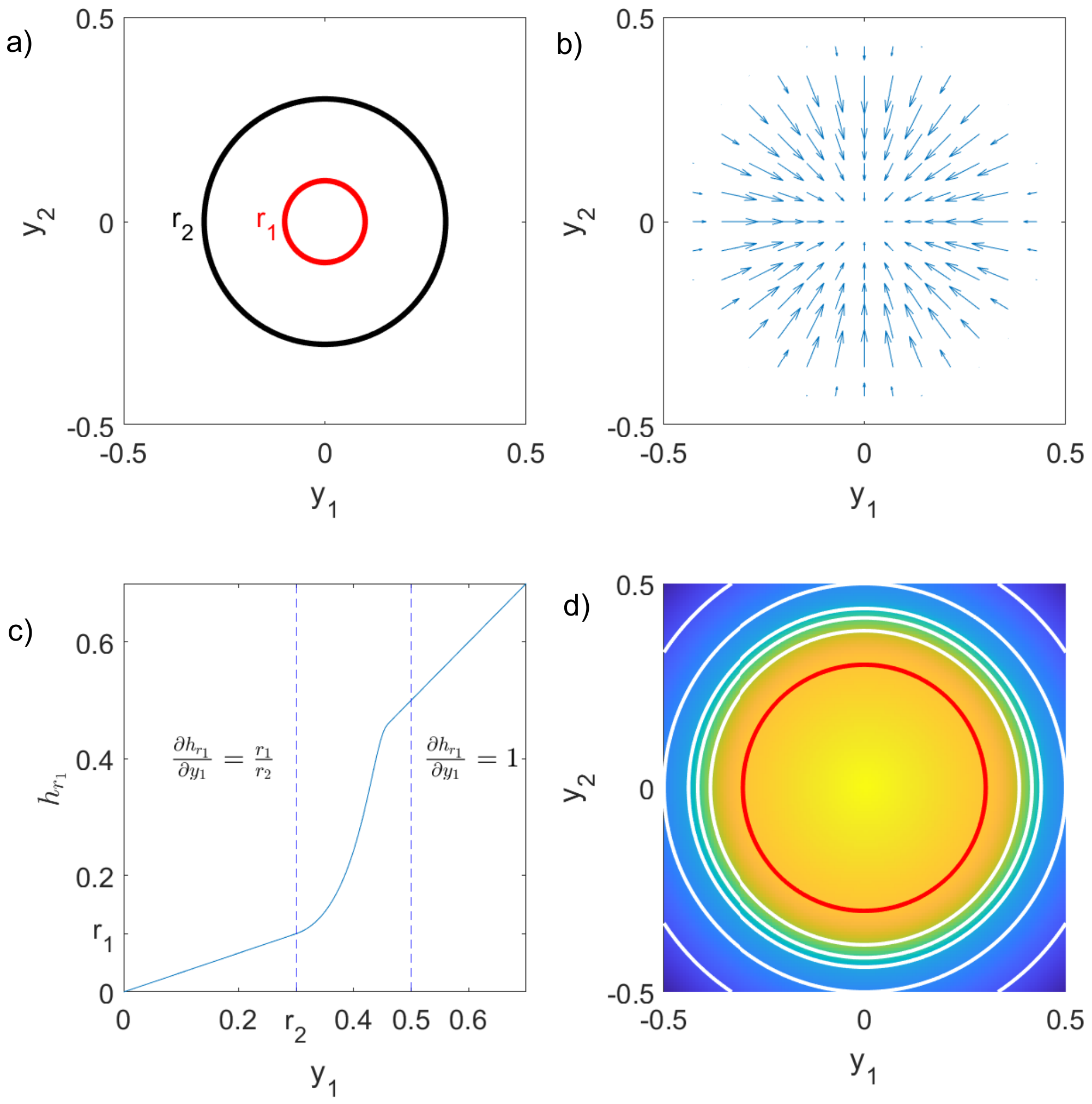}
    \caption{Visualization of diffeomorphism~(\ref{EQ:DiffeoCircle}) for $r_2=0.3,\; r_1=0.1$: Image a) illustrates the related circles ($r_2$-black, $r_1$-red) posing the interior boundary of the domain. In image b), the displacement field is shown, i.e.\ $h_{r_1}-\text{id}_Y$. As enforced by the interpolation function $\xi$ in~(\ref{EQ:DiffeoCircle}), the displacement smoothly vanishes close to the exterior boundary $\partial Y$. In c), the graph of $h_{r_1}$, cf. (\ref{EQ:DiffeoCircle}), is shown along the $y_1$ axis illustrating the three different sections (uniform contraction - transition - identity).  
    Figure d) displays the pullback $h_{r_1}^*(\Phi)$ with $\Phi(y_1,y_2)=r_1-||(y_1,y_2)||_2$. Contour lines uniformly spaced in increments of $0.1$ are added in white. The zero level-set of $h_{r_1}^*(\Phi)$ highlighted in red corresponds to a~circle of radius $r_2$.  }
    \label{fig:DiffeoCircles}
\end{figure}
\end{Example}

In the following, we prove differentiability of the diffusion tensors along such 1-parametric curves. More precisely, $C^m$-mappings $h:(-S,S) \to \text{Diff}_\Box^{2,1}(\bar{Y})$ will be considered.
Accordingly, we switch the above setting to the following:
\begin{align*}
    F&:  H^2_{\#,0} (\mathcal{P}) \times (-S,S) \to L_0^2(\mathcal{P}) \times H^\frac{1}{2}(\partial^\text{int}\mathcal{P}),\\
    F & (u,s) = \left(h_s^*\Delta h_s^{*-1} u - \fint\limits_\mathcal{P} h_s^*\Delta h_s^{*-1} u \; dy,\; h_s^* \text{Tr}_{h_s(\mathcal{P})}\left(\nu_{h_s(\mathcal{P})}  \cdot (\nabla h_s^{*-1} u - e_1)  \right) \right), \nonumber
\end{align*}
tracking along a~fixed path of diffeomorphisms in comparison to~(\ref{EQ:F}).
In order to obtain higher differentiability of the resolution function $g$ in Theorem~\ref{THEOREM1}, higher regularity of $F$ needs to be established. Revisiting the calculations performed in (\ref{EQ: ContinuityF1}) we immediately see a~transfer of regularity to $F$ with respect to the second variable. As the mapping is linear with respect to the first variable, we again obtain $C^m$-regularity for $F$. More precisely, the following theorem holds extending the smooth dependence results for simple parametric families of shapes as derived in~\cite{Schulz2017, eden2021multiscale}.

\begin{Theorem} \textit{Smooth dependence of $\ID$}\\
Consider a~$C^m$-curve $s\mapsto h_s$ of $\text{Diff}_{\Box}^{2,1}(\bar{Y})$-embeddings
and $h_0=\text{id}_{\bar{Y}}$ for $m\geq 1$. Assume a~$C^{2,1}$ open set $Y\setminus\bar{\mathcal{P}} \subset Y$ being compactly contained in $Y$ and $u\in H^2_{\#,0}(\mathcal{P})$ such that \mbox{$F(u,0)= (0,0)$}, i.e.\ $u$ is a~solution to problem (\ref{EQ: DiffWeakForm}). Then there exists a~neighborhood $V$ of zero and a~$C^m$-function $g:V\to H^2_{\#,0}(\mathcal{P})$, $g(0)=u$, such that
\begin{align*}
 F(g(s),s)= (0,0), \quad \forall s\in V.    
\end{align*}
Particularly, $h_s^{*-1}g(s)$ solves (\ref{EQ: DiffWeakForm}) up to an~additive constant on $h(\mathcal{P})$ for all $s\in V$.
\label{THEOREM1b}
\end{Theorem}

\begin{proof}
This is again an~immediate consequence of the implicit function theorem for Banach spaces as given in \cite{Valent}.
\end{proof}

Again, we can leverage this regularity to the diffusivity tensors (\ref{Def: DiffusionTensor}). 
\begin{Korollar}
\label{Kor: DiffDiffbar}
Under the assumptions of Theorem \ref{THEOREM1b} the mapping 
\begin{align*}
    R: V \to \mathbb{R} ,\quad  s\mapsto \int_{h_s(\mathcal{P})}  \nabla h_s^{*-1}g(s) \; dy
\end{align*} 
is $m$-times continuously differentiable. Therefore, the diffusion tensor~(\ref{Def: DiffusionTensor}) depends $C^m$-regularly on variations of the domain $\mathcal{P}$ along a~path of diffeomorphisms of specified regularity. 
\end{Korollar}
\begin{proof}
Since the map $s \mapsto h_s$ is $C^m$-regular, we establish the same degree of Fréchet differentiability in the spatial derivatives $s \mapsto \nabla h_s$ as mappings $(-S,S) \to C^1(\bar{Y},\bar{Y})$ with respect to the corresponding norms. Revisiting the calculations performed in (\ref{EQ:KorDiffConti}) and using the product rule for Fréchet differentiable functions, we conclude the assertion.
\end{proof}

\begin{Bem}
Due to the mathematical structure of the problem, Theorem \ref{THEOREM1b} holds analogously for families of diffeomorphisms that are parameterizable by a~finite number of parameters. As such, the real-valued order parameter $s$ can be replaced by its vector-valued analogue, allowing for more sophisticated couplings and more complex geometries. A~natural field of application is posed by two-mineral-phase solids, cf.~\cite{Gaerttner2020bPreprint}, where the two interacting phases obey their distinct evolution laws.   
\end{Bem}

\subsection{Existence for partial coupling}
\label{SEC:DiffExist1Sided}

In this section, we present an~existence result for strong local-in-time solutions to~(\ref{EQ:transportDIFF}) under the assumption of no back-coupling from the macro- to the micro-scale, cf.\ Figure~\ref{fig:Coupling}. That is, we assume the evolution of the underlying pore geometry to be known a-priori. As such, the treatment of the problem is accessible more easily in comparison to the fully coupled system which we will discuss in Section~\ref{SEC:DiffExist2Sided}. In~\cite{Schulz2017}, a~similar fully coupled problem was solved under the assumption of effective parameters being parameterizable by the porosity $\phi$ in a~smooth and a-priori known way. As opposed to this restriction, we will introduce a~new order parameter $s$ which corresponds to the parametrization in the path of diffeomorphisms used to describe evolved initial geometries. Particularly, this approach allows for the apposite description of multiple geometries which admit the same porosity. Due to the identical structure of the model, we will use the methods of~\cite{Schulz2017} for our following analysis.

\begin{figure}[!h]
    \centering
    \includegraphics[width=1\textwidth]{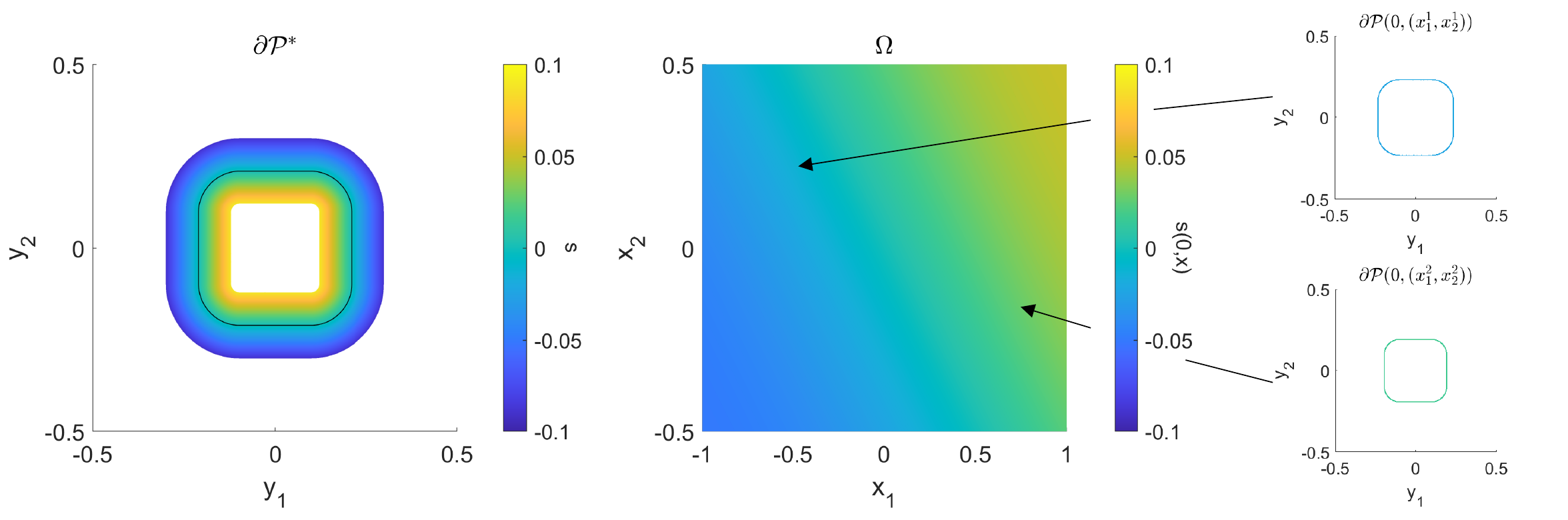}
    \caption{Left: Master unit-cell $Y^*$ with interior boundary $\partial^\text{int}\mathcal{P}^*$ in black corresponding to $s=0$. Each color corresponds to the interior boundary of a~deformed cell which is reachable from $\mathcal{P}^*$ along a~path of diffeomorphisms parameterized by $s$. Right: Macroscopic domain $\Omega$ colored according to the initial underlying geometry displayed in the left image. For the two exemplary macroscopic points $x_1,x_2\in\Omega$ the associated microscopic geometry is displayed. }
    \label{fig:MasterCell}
\end{figure}

To this point, we considered an~initial geometry $\mathcal{P}$ of class $C^{2,1}$ and a~parameterized path of diffeomorphisms $h:(-S,S)\to \text{Diff}_\Box^{2,1}(\bar{Y})$ of class $C^1$ with $S\in\mathbb{R}^+$ and $h_0 = \text{id}_{\bar{Y}}$.  In order to allow for spatial variations of effective parameters, we realize the prescription of the geometry evolution by specifying $s: \Omega_T \to (-S,S)$. That is, we assume the state of each microscopic unit-cell $Y(t,x)$ to be given as the state of a~single master unit-cell $Y^*$ at time-parameter $s(t,x)$. The corresponding setup is visualized in Figure~\ref{fig:MasterCell} illustrating the assignment of initial conditions on $\Omega$. Accordingly, the effective parameters of~(\ref{EQ:transportDIFF}) are given by $\phi(s),\;  \sigma(s),\; \ID (s)$. Note that we perform the necessary redefinition of the functions $\phi,\;  \sigma,\; \ID$ to mappings from an open interval of~$\mathbb{R}$ without change of notation. We furthermore restrict to the non-degenerative case, i.e.\ we assume:
\begin{align}
\label{AssumptionsOneSidedDiff}
   \forall s\in (-S,S): 0< \phi(s) <1,  \quad \sigma(s) >0, \quad \ID(s)>0,
\end{align}
where the last inequality hold in the sense of matrices (Loewner partial ordering). This assumption is naturally fulfilled for small times $t>0$ by suitably prepared initial conditions. Moreover, we restrict our consideration to locally Lipschitz reaction rates $f(c)$ generalizing the linear reaction rates prescribed in~\cite{Schulz2017}. Finally, we state the following anisotropic Sobolev spaces
\begin{align*}
    \mathcal{X}_1:=W^{1,2}_r(\Omega_T) = L^r(0,T;W^{2,r}(\Omega)) \cap W^{1,r}(0,T;L^r(\Omega)),
\end{align*}
with $r>d+2$ which play a~crucial role in the subsequent existence result. Following the major steps of~\cite{Schulz2017}, we have:

\begin{Theorem} \textit{Existence of strong solutions, partial coupling, diffusive transport}\\
Let $\Omega\subset\mathbb{R}^d,\; d\in\{2,3\}$, be a~$C^2$-domain, $r>d+2$ with initial conditions $c_0\in W^{2-\frac{2}{r},r}(\Omega)$ and Dirichlet boundary conditions $C_0 \in W^{1-\frac{1}{2r}, 2-\frac{1}{r}}_r(\partial \Omega_T)$ being compatible in the sense of $C_0(0,\cdot)=c_0$ on $\partial\Omega$. Moreover, let the evolution of the pore-space geometry be given by an~order parameter $s\in C^1(\overline{\Omega_{T_1}})$, $s(t,x)\in (-S+\epsilon,S-\epsilon),\; \epsilon>0$, \mbox{$\forall (t,x)\in\Omega_T$}, and a~path $h\in C^1(-S,S;\text{Diff}_{\Box}^{2,1}(\bar{Y}))$ of diffeomorphisms such that $h_0 = \text{id}_{\bar{Y}}$. Assume the initial inclusion $Y^*\setminus \bar{\mathcal{P}} \subset Y^*$ to be a~compactly contained $C^{2,1}$ open set. Let~(\ref{AssumptionsOneSidedDiff}) hold true and $f$ be locally Lipschitz. Then there exists a~time $0<T\leq T_1$ such that (\ref{EQ:transportDIFF}) admits a~unique solution $c\in \mathcal{X}_1$. 
\label{Theorem:DiffOneSideExist}
\end{Theorem}

\begin{proof}
First, we note that using the results of Theorem~\ref{THEOREM1b} the mapping $\ID: (-S,S) \to \mathbb{R}^{d,d}$ is of class $C^1$ and accordingly $\ID\circ s \in C^1(\overline{\Omega_{T_1}})$. Similarly, we obtain $\phi\in C^1((-S,S))$ as a~consequence of the following representation
\begin{align*}
    \phi(s) =\int\limits_{h_s(\mathcal{P})} 1\; dy = \int\limits_{\mathcal{P}} \lvert\text{det} \left(\nabla h_s\right)\rvert\; dx
\end{align*}
using the change of variables theorem. In order to establish regularity for $\sigma$, we consider $h(\partial^\text{int}\mathcal{P})$ as an~evolving manifold. Let $\varphi:U\subset \mathbb{R}^{d-1}\to \mathbb{R}^d$ be a~suitable local parameterization of $\partial^\text{int}\mathcal{P}$. Then, $\hat{\varphi}: (-S,S) \times U \to \mathbb{R}^d$ defined as
\begin{align*}
    \hat{\varphi}(s,x) := h_s(\varphi(x))
\end{align*}
is a~parameterization of $h_s(\partial^\text{int}\mathcal{P})$ for each $s\in (-S,S)$. Using the regularity of $\hat{\varphi}$ we infer continuity of the mapping
\begin{align*}
    s \mapsto \int\limits_{h_s(\varphi(U))} 1 \; d\sigma = \int\limits_{U} \sqrt{\text{det}\left(\nabla \hat{\varphi} (\nabla \hat{\varphi} )^T \right) }  \; dx
\end{align*}
which translates to the continuity of $\sigma$ by using a~partition of unity subordinate to the domains of parameterization, cf.~\cite{RiemannManifolds}.
As a~result of the continuity of all effective parameters with respect to $s$, there exists a~$\delta \in (0,1)$ such that:
\begin{align}
\label{EQ:BoundsExist1}
   \forall s\in (-S+\epsilon,S-\epsilon) : \; \delta< \phi(s) <1-\delta,  \quad \sigma(s) > \delta, \quad \ID(s)>\delta \mathds{1}_d.
\end{align}
Note that the continuity of the eigenvalues is inherited from the continuity of $\ID$, cf.~\cite{KatoPerturbation}. Following the technique presented in~\cite{Schulz2017}, the proof is now based on Schauder's fixed point theorem applied to the set
\begin{align*}
    \mathcal{K}_1 = \left\lbrace c\in \mathcal{X}_1 : ||c||_{\mathcal{X}_1}\leq K \right\rbrace
\end{align*}
for a~constant $K\geq 1$ chosen appropriately later. Apparently, $\mathcal{K}_1$ is a~convex, closed and bounded subset of $\mathcal{X}_1$. In order to apply standard linear solution theory of parabolic equations, we rewrite equation~(\ref{EQ:transportDIFF}) in the following fixed-point form:
\begin{align}
\label{EQ: FixedPointTransDiff}
    \partial_t c -\nabla \cdot\left( \frac{\ID(s)}{\phi(s)} \nabla c\right) = -\frac{\partial_t [\phi(s)]}{\phi(s)} \tilde{c} + \frac{\sigma(s) }{\phi(s)} f(\tilde{c}) + \frac{\ID(s)}{\phi(s)^2} \nabla [\phi(s)] \cdot \nabla \tilde{c}.
\end{align}

Now consider the mapping $\mathcal{F}_1:\mathcal{K}_1\to L^{r}(\Omega_T)$ mapping a~concentration $\tilde{c}\in \mathcal{K}_1$ to the right-hand side of~(\ref{EQ: FixedPointTransDiff}). Using the compact embeddings 
\begin{align}
\label{CompactEmbedX1}
\mathcal{X}_1 \hookrightarrow W^{\frac{3}{4},\frac{3}{2}}_{2r}(\Omega_T), \quad \mathcal{X}_1 \hookrightarrow C^{\frac{1}{2},1}(\overline{\Omega_T}),  
\end{align}
$\mathcal{F}_1$ shows to be compact, cf.~\cite{Schulz2017}. Furthermore, the parabolic theory of Theorem~\ref{TheoremLadyzenkaya} delivers a~continuous solution operator $\mathcal{F}_2:  L^{r}(\Omega_T) \to W^{1,2}_r(\Omega_T)$ to~(\ref{EQ: FixedPointTransDiff}). More precisely, we apply Theorem~\ref{TheoremLadyzenkaya} to the above prescribed initial and boundary conditions with coefficients defined as
\begin{align*}
    a_{i,j}(t,x) = \frac{\ID_{i,j}(s(t,x))}{ \phi(s(t,x))},\quad a_i(t,x) = -\sum\limits_{j=1}^d \partial_j \left( \frac{\ID_{i,j}(s(t,x))}{ \phi(s(t,x))}\right) , \quad a(t,x) = 0
\end{align*}
and source term $f$ according to the right-hand side of (\ref{EQ: FixedPointTransDiff}).

As such, for a~given $\tilde{c}\in \mathcal{X}_1$, we have 
\begin{align*}
    ||c||_{\mathcal{X}_1} \leq C_p \left(\text{DC}+
    \left \lVert-\frac{\partial_t [\phi(s)]}{\phi(s)} \tilde{c} + \frac{\sigma(s)}{\phi(s)} f(\tilde{c}) + \frac{\ID(s)}{\phi(s)^2} \nabla [\phi(s)] \cdot \nabla \tilde{c}\right\rVert_{L^r(\Omega_T)}  \right),
\end{align*}
abbreviating the contributions from initial and boundary data by
\begin{align*}
\text{DC}:=||c_0||_{W^{2-\frac{2}{r},r} (\Omega)} + ||C_0||_{W^{1-\frac{1}{2r}, 2-\frac{1}{r}}_r(\partial \Omega_T)}.    
\end{align*}
Due to $f$ being locally Lipschitz, there exists a~monotone function $\tilde{f}$ satisfying 
\begin{align}
\label{EQ:MonotnoneBound}
\tilde{f}: [0,\infty) \to \mathbb{R}, \quad    |f(x)| \leq \tilde{f}(|x|),
\end{align}
Similar to the estimates established in~\cite{Schulz2017}, we obtain using Hölder's inequality and~(\ref{EQ:MonotnoneBound})

\begin{align*}
||c||_{\mathcal{X}_1} &\leq C_p \left(\text{DC} + C_s T^\frac{1}{r} |\Omega|^\frac{1}{r}\left(||\tilde{c}(t)||_{L^\infty(\Omega_T)}  + \tilde{f}(||\tilde{c}||_{L^\infty (\Omega)}) \right) \right)   \\
&+ C_p C_s \left(T^\frac{1}{2r} ||\nabla [\phi(s)] ||_{L^\infty(0,T; L^{2r}(\Omega))}   ||\nabla \tilde{c} ||_{L^{2r}(0,T; L^{2r}(\Omega))} \right),  \nonumber
\end{align*}
with constant $C_s$ depending on the bounds of $\sigma, \phi,\partial_t \phi,\ID$, cf.~(\ref{EQ:BoundsExist1}).
By the embeddings~(\ref{CompactEmbedX1}), every appearing norm of $\tilde{c}$ is bounded by a~multiple of $||\tilde{c}||_{\mathcal{X}_1}$ and therefore by a~multiple of $K$. As such, we obtain the self-mapping property of $\mathcal{F}=\mathcal{F}_2\circ\mathcal{F}_1: \mathcal{K}_1 \to \mathcal{K}_1$ for sufficiently large $K$ and sufficiently small $T$, finishing the existence proof. Uniqueness follows analogously to~\cite{Schulz2017} by consider two solutions $c_1,c_2\in \mathcal{X}_1$ as well as their difference $\bar{c}=c_2-c_1$. Subtracting both associated equations in fixed-point form~(\ref{EQ: FixedPointTransDiff}), we obtain
\begin{align}
\label{EQ:DiffUniqueness}
    \partial_t \bar{c}  -\nabla \cdot\left( \frac{\ID(s)}{\phi(s)} \nabla \bar{c}\right) = -\frac{\partial_t [\phi(s)]}{\phi(s)} \bar{c} + \frac{\sigma(s) }{\phi(s)} \left(f(c_2)-f(c_1)\right) + \frac{\ID(s)}{\phi(s)^2} \nabla [\phi(s)] \cdot \nabla \bar{c}.
\end{align}
Testing~(\ref{EQ:DiffUniqueness}) with $\bar{c}$ and estimating the right-hand side with Hölder's and Young's inequality shows
\begin{align}
\label{EQ:DiffUniqueness2}
    \frac{1}{2} \partial_t ||\bar{c}||^2_{L^2(\Omega)}&\leq \left( \left\lVert \frac{\partial_t [\phi(s)]}{\phi(s)} \right\rVert_{L^\infty(\Omega)}  + L \left\lVert\frac{\sigma(s) }{\phi(s)} \right\rVert_{L^\infty(\Omega)}  + C(\epsilon)\left\lVert\frac{\ID(s)}{\phi(s)^2} \nabla [\phi(s)] \right\rVert_{L^\infty(\Omega)}\right)  ||\bar{c}||^2_{L^2(\Omega)}\nonumber \\
    &-  \int\limits_\Omega \frac{\ID(s)}{\phi(s)}\nabla \bar{c} \cdot \nabla \bar{c} \; dx + \epsilon \left\lVert\frac{\ID(s)}{\phi(s)^2} \nabla [\phi(s)] \right\rVert_{L^\infty(\Omega)} ||\nabla \bar{c}||^2_{L^2(\Omega)}, 
\end{align}
where $L$ denotes the Lipschitz constant of $f$ with respect to the compact interval 
\begin{align*}
\left[-\max\limits_{i\in\{1,2\}} \left\lbrace||c_i||_{C^0(\overline{\Omega_T})}\right\rbrace,\; \max\limits_{i\in\{1,2\}} \left\lbrace||c_i||_{C^0(\overline{\Omega_T})}\right\rbrace\right]    .
\end{align*}
By the uniform coercivity of $\ID$, we can absorb the last addend of~(\ref{EQ:DiffUniqueness2}) into the diffusion term for sufficiently small $\epsilon>0$. Uniqueness now follows from Gronwall's inequality. 
\end{proof}

\subsection{Level-set equation induced diffeomorphisms}
\label{SEC:FlowInducedDiffeomorphisms}
The smooth dependence results derived in Section~\ref{SEC:diff} are based on geometry deformation by smooth paths of diffeomorphisms. In the following, we investigate under which conditions on the initial geometry and normal velocity field alterations performed by the level-set equation~(\ref{EQ:levelSet}) induce such paths. To do so, we make use of the method of characteristics. As a~result, we can replace the assumption of a~prescribed path of diffeomorphisms in Theorem~\ref{Theorem:DiffOneSideExist} by a~prescribed normal velocity field $v_n(t,x,y)$ and let the geometry evolve according to the level-set equation which is a~much more natural setup from the viewpoint of applications. At first, we must fix the class of real-valued functions whose level-sets are guaranteed to be smooth submanifolds of codimension one:

\begin{mydef}\textit{Regular level-set function} \\
Let $\Gamma \subset Y$ denote the boundary of an~open set of class $C^{k,\alpha}$, $k\geq 2$, compactly contained in $Y$. Then we call a~function $\Phi \in C^{k,\alpha}(Y)$ \textit{regular level-set function} associated to $\Gamma$ iff $\Gamma = \{y\in Y:\; \Phi(y) = 0\}$ and $\nabla \Phi = -\nu$ on $\Gamma$.
\label{DEF:regularLevelSetFunction}
\end{mydef}

\begin{Bem}
For every manifold $\Gamma$ as given in Definition~\ref{DEF:regularLevelSetFunction} there exists an~associated regular level-set function, cf.~\cite{Henry} Chapter~1 or Theorem~\ref{TheoremTubNeigh}.
\end{Bem}

In order to apply the theory for non-linear first-order PDEs, we rewrite the level-set equation~(\ref{EQ:levelSet}) in the form $F(\vec{x},\Phi,D\Phi)=0$ with $F:(\vec{x},z,\vec{p})\mapsto\mathbb{R}$. Note that in this notation, $\vec{x}=(x,t)$, $\vec{p}=(p,p_{d+1})$ corresponds to the spatial and temporal variable, $D=(\nabla_x,\partial_t)$ denotes the related differential operator. Apparently, we have
\begin{align*}
    F(\vec{x},z,\vec{p}) = p_{d+1}+v_n(\vec{x}) |p|. 
\end{align*}
with the normal interface velocity $v_n$ of~(\ref{EQ:levelSet}). Prescribing $v_n$ smooth such that it vanishes in a~neighborhood of $\{|p|=0 \}$, we have \mbox{$F\in C^2(\mathbb{R}^{2d+3},\mathbb{R})$}.
Switching to the characteristic system of ODEs, the equations read
\begin{align}
\label{EQ:characteristics}
    \dot{\vec{p}}(s) &=- \partial_{\vec{x}} F (\vec{x}(s),z(s),\vec{p}(s)) - \partial_{z} F (\vec{x}(s),z(s),\vec{p}(s)) \vec{p}(s)  = -Dv_n(\vec{x}) |p|, \nonumber \\
    \dot{z}(s) &= \partial_{\vec{p}} F(\vec{x}(s),z(s),\vec{p}(s))\cdot \vec{p}(s) = v_n(\vec{x}) \frac{p}{|p|}\cdot p + p_{d+1} = 0,\\
    \dot{\vec{x}}(s) &= \partial_{\vec{p}} F (\vec{x}(s),z(s),\vec{p}(s)) = \left(v_n(\vec{x}) \frac{p}{|p|}, 1\right), \nonumber
\end{align}
with initial conditions
\begin{align}
\label{EQ:characteristicsIC}
    \vec{p}(0) = \left(\nabla \Phi(0,y), -v_n(y)|\nabla \Phi(0,y)| \right), \quad z(0) = \Phi(0,y), \quad \vec{x}(0)=(y,0).
\end{align}
According to~(\ref{EQ:characteristics}), the projected characteristics of a~solution to~(\ref{EQ:levelSet}) move in normal direction to the interface with speed $v_n(\vec{x})$. Furthermore, the function value of $\Phi$ remains constant along trajectories. As such, the zero-level set describing the position of the fluid-solid interfaces is transported along $x$. Note that every admissible point of the characteristic ODE system is non-characteristic, i.e.\ the characteristics admit a strictly monotone distance to the set of prescribed initial data $\bar{Y}\times \{0\}$. As such, equation~(\ref{EQ:levelSet}) admits a~unique local-in-time $C^2$-solution by standard theory~\cite{Evans}. Furthermore, the parameterized trajectories associated to~$x$ induce a~smooth path of diffeomorphisms and the artificial parameter $s$ coincides with the actual physical time $t$. More precisely, we have the following statement:

\begin{Lemma}
\label{Lemma:LevelSetDiffeo}
Let a~normal velocity field $v_n\in C^{4,1}(\bar{Y},\mathbb{R})$ be given. Furthermore, let a~regular level-set function $\Phi_0\in C^4(\bar{Y},\mathbb{R})$ be given, such that \mbox{$\bar{\mathcal{P}}=\left\lbrace \Phi_0\leq 0 \right\rbrace$} with $Y\setminus \mathcal{P}$ being compact in $Y$. Then the local-in-time solution to the level-set equation~(\ref{EQ:levelSet}) with respect to initial conditions $\Phi_0$ induces a~$C^1$-path $h$ in $\text{Diff}_\Box^{2,1}(\bar{Y})$ such that \mbox{$h_t(\bar{\mathcal{P}}) = \left\lbrace \Phi(t,\cdot)\leq 0 \right\rbrace$} for all times $t$ sufficiently small.
\end{Lemma}

\begin{proof}
In the following, we investigate the regularity of the trajectories associated to~$\vec{x}$ solving the systems of ODEs~(\ref{EQ:characteristics}). More precisely, we consider the smaller system in the spatial variables $x$ and $p$ which is closed due to the time-independence of $v_n$. As a~suitable Banach space for $(x,p)$, we introduce
\begin{align*}
    \mathcal{X} = (C^3(\bar{Y}, \mathbb{R}))^{d}  \times (C^3(\bar{Y}, \mathbb{R}))^{d}.
\end{align*}
Then the reduced initial conditions~(\ref{EQ:characteristicsIC}) are element of $\mathcal{X}$.
It is straightforward to check that the structure function $f:\mathcal{X} \supset V \to \mathcal{X}$ of the reduced ODE system
\begin{align*}
    f(x,p) = \left(v_n(x) \frac{p}{|p|} ,\;-\nabla v_n(x) |p| \right)
\end{align*}
is well-defined and locally Lipschitz continuous with respect to the norm $|.|_{\mathcal{X}}$ for a small neighborhood $V$ of $(x(0),p(0))$, as long as $v_n = 0$ in a~neighborhood of \mbox{$\{y\in Y :\; |\nabla \Phi_0(y)| = 0\}$}. Due to the tubular neighborhood theorem (Theorem~\ref{TheoremTubNeigh}), this can be achieved by modifying $v_n$ appropriately without changing the solution locally at $\partial^\text{int} \mathcal{P}$. More precisely, we force $v_n$ to zero away from a~neighborhood of $\partial^\text{int} \mathcal{P}$ and within a~neighborhood of $\partial Y$ in a~smooth manner. By Picard-Lindelöf theorem, there exists $T>0$ such that the ODE system admits a~unique solution $(x,p)\in C^1(0,T; \mathcal{X})$. Due to the convexity of the domain, we obtain $x\in C^1(0,T; C^{2,1}(\bar{Y},\mathbb{R}^d))$. By the choice of $v_n$, we have $\text{im}(x(s)) = \bar{Y}$ for all $s\in(0,T)$. Since the initial conditions fulfill $x(0)=\text{id}_{\bar{Y}}\in \text{Diff}^{2,1}(\bar{Y})$, the diffeomorphism property for $x(s)$ is obtained in a~possibly reduced time interval $(0,T)$. Consequently, $x\in C^1(0,T; \text{Diff}_\Box^{2,1}(\bar{Y}))$. Since the value of $\Phi$ is constant along the trajectories associated to~$x$, we finally see $h_t(\bar{\mathcal{P}}) = \left\lbrace \Phi(t,\cdot)\leq 0 \right\rbrace$.
\end{proof}

In order to underline the power of the above Lemma~\ref{Lemma:LevelSetDiffeo}, let us reconsider the case of contracting and expanding circles as in Example~\ref{Ex:DiffeoCircle}. Previously, a~suitable family of diffeomorphisms had to be constructed explicitly in~(\ref{EQ:DiffeoCircle}) to apply Theorem~\ref{THEOREM1b}. With the help of Lemma~\ref{Lemma:LevelSetDiffeo} it is now sufficient to check that the function
\begin{align*}
    \Phi_r(y) = r - ||y||_2
\end{align*}
has a~zero-level-set of a~circle of radius $r<0.5$ and can be smoothed locally around $0$ to fulfill $\Phi_{r,0} \in C^4(\bar{Y})$. The statement follows by applying Lemma~\ref{Lemma:LevelSetDiffeo} to $\Phi_{r,0}$ and $v_n \equiv 1$. 

\subsection{Existence for full coupling}
\label{SEC:DiffExist2Sided}

In geoscientific applications of our model, the evolution of the microscopic geometry of the porous medium is not known in advance but depends on the time-dependent macroscopic concentration field. Accordingly, this chapter is dedicated to the local-in-time existence of strong solutions to a~model where the geometry evolution is described by the level-set equation~(\ref{EQ:levelSet}).
For clarity, we restate the model equations:
\begin{align*}
    \phi \partial_t c - \nabla\cdot(\mathbb{D}\nabla c) &=   \sigma f(c) - \partial_t \phi c &&\quad \text{in } (0,T)\times \Omega, \\
    \frac{\partial \Phi}{\partial t} + v_n |\nabla_y \Phi| &= 0, &&\quad \text{in }  (0, T)\times \Omega \times Y,  \nonumber
\end{align*}
supplemented by boundary and initial conditions as well as diffusion cell-problem~(\ref{Def: DiffusionTensor}), where the level-set normal interface velocity is given by
\begin{align}
\label{Def:VnormalSimple}
    v_n(t,x) = f(c(t,x)),
\end{align}
i.e.\ depending on the solution of the macroscopic concentration. In simplification of~(\ref{DEF:NormalVel}), we consider a~uniform velocity within each unit-cell to facilitate the establishment of higher regularity for quantities derived from the geometry. Furthermore, we restrict to scenarios where the porosity can be used as the natural order parameter uniquely characterizing the geometry state which is the typical case in dissolution/precipitation processes. Moreover, we assume linear reaction rates. By doing so, we retract to the setting investigated in~\cite{Schulz2017} enabling us to use respective results stated therein. Additionally, we again consider a~single evolving master unit-cell $Y^*$ as in Section~\ref{SEC:DiffExist1Sided} and introduce spatial inhomogeneity in $\Omega$ by choosing different evolution states of $Y^*$ as the initial condition.

In the following, we approach existence of solutions $(c,\Phi)\in (\mathcal{X}_1, \Omega_x\times C^1((0,T)\times Y))$ to the fully coupled model described above, where the solution space with respect to $\Phi$ is defined as
\begin{align*}
    \Omega_x\times C^1((0,T)\times Y) = \left\lbrace \Phi: (0, T)\times \Omega \times Y \to \mathbb{R}: \; \forall x\in\Omega \quad \Phi(\cdot,x,\cdot) \in C^1((0,T)\times Y)  \right\rbrace.
\end{align*}
To do so, we rewrite the given problem in the variables $(c,\phi)\in\mathcal{X}_1^2$. As such, we can apply Theorem~\ref{TheoremSchulz} to obtain local-in-time existence in the variables $(c,\phi)$. In order to do so, the effective parameters $\ID,\sigma$ must be expressed as functions of $\phi$. As we will see, this relation is independent of the concentration solution $c$ and is therefore determined by the level-set initial condition alone. Given the solution $(c,\phi)$ of the transformed system we construct the solution $(c,\Phi)$ to our model of consideration.

Let us consider the geometry evolution being preliminarily parameterized by the time $t$ of the level-set equation. Using the properties of (\ref{Def:VnormalSimple}), we can reparametrize all effective parameters as a~function of $\phi$ by setting:
\begin{align}
\label{EQ:TrafoWRTphi}
    \hat{\sigma} = {\sigma} \circ \phi^{-1}, \quad  \hat{\ID} = {\ID} \circ \phi^{-1}.
\end{align}
The main difficulty in proving existence for the bilaterally coupled system is establishing sufficiently high regularity in the coefficients $\hat{\sigma}, \hat{\ID}$ as required to apply Theorem~\ref{TheoremSchulz}.

First, we notice that the mapping $\phi \mapsto \hat{\sigma}, \hat{\ID}$ is independent of the particular choice of a~normal velocity $v_n$ in (\ref{Def:VnormalSimple}) in the set of y-independent functions which can be seen as follows: Let $\Phi_1$ solve~(\ref{EQ:levelSet}) with $v_n \equiv 1$. Then $\Phi(t,x) = \Phi_1(V(t),x)$
solves~(\ref{EQ:levelSet}) with $v_n =v$ and $V$ being the primitive of $v$. As such, solutions are solely rescaled with respect to the time variable but not with respect to space, i.e.\ the individual geometries relating to $\Phi_1(s,\cdot)$ remain unaffected. It is therefore sufficient to consider the case $v_n\equiv 1$, where, due to
\begin{align}
\label{EQ:PhiSigma}
    \partial_s \phi(s) =  -\sigma(s),
\end{align}
$\phi$ is strictly monotonically increasing. This fact justifies the transformations introduced in~(\ref{EQ:TrafoWRTphi}). As such, the regularity of $\phi \mapsto \hat{\sigma}, \hat{\ID}$ follows immediately from the regularity of $\phi, \sigma, \ID$ which we investigate using the same means as in Theorem~\ref{Theorem:DiffOneSideExist}.

Let $\varphi:U\subset \mathbb{R}^{d-1}\to \mathbb{R}^d$ be a~suitable local parameterization of $\partial^\text{int}\mathcal{P}$. Then, $\hat{\varphi}: (-S,S)\times U \to \mathbb{R}^d$ defined as
\begin{align}
\label{EQ:NormalEvo}
    \hat{\varphi}(s,x) := \varphi(x) + s\nu(\varphi(x))
\end{align}
is a~local parameterization of the deformed interface at times $s\in (-S,S)$ using the results of Section~\ref{SEC:FlowInducedDiffeomorphisms}. Note that $\hat{\varphi}(s,\cdot)$ is a~bijection for $s$ close to zero and in case of $\Phi_0 \in C^4(Y)$, the map~(\ref{EQ:NormalEvo}) is of class $C^3$, cf.~Theorem~\ref{TheoremTubNeigh}. Using the regularity of $\hat{\varphi}$ we infer the mapping
\begin{align*}
    s \mapsto  \int\limits_{U} \sqrt{\text{det}\left(\nabla \hat{\varphi} (\nabla \hat{\varphi} )^T \right) }  \; dx
\end{align*}
to be of class $C^2$, which translates to $\sigma(s)$ by using a~partition of unity, cf.~\cite{RiemannManifolds}. By~(\ref{EQ:PhiSigma}) we also obtain $\phi \in C^3((-S,S))$. As such, we finally conclude $\hat{\ID} \in C^1((\phi_\text{min},\phi_\text{max})) $, $\hat{\sigma} \in C^2((\phi_\text{min},\phi_\text{max})) $ for sufficiently narrow bounds $\phi_\text{min}<\phi^0<\phi_\text{max}$ around the initial condition. Summing up the observations of this section, we obtain the following existence result using Theorem~\ref{TheoremSchulz}.

\begin{Theorem} \textit{Existence of strong solutions, full coupling, diffusive transport}\\
Let $\Omega\subset\mathbb{R}^d,\; d\in\{2,3\}$, be a~$C^2$-domain, $r>d+2$ with initial conditions $c_0\in W^{2-\frac{2}{r},r}(\Omega)$, $c_0\geq 0$, fulfilling $c_0=0$ on $\partial\Omega$. Furthermore, consider an~initial inclusion compactly contained in $Y^*$ given by a~regular $C^4(\bar{Y}^*)$ level-set function evolving according to the level-set equation~(\ref{EQ:levelSet}) such that solutions for $v_n \equiv \pm1$ cover $\phi \in (\phi_\text{min},\phi_\text{max}) \subset (0,1)$. Let $\phi^0\in W^{2,r}(\Omega)$, $\phi^0(x)\in (\phi_\text{min}+\epsilon,\phi_\text{max}-\epsilon)$ for some $\epsilon>0$, $x\in \Omega$ and the reaction rate given as $f(c)=c$. Then there exists a~local-in-time solution $(c,\Phi)\in (\mathcal{X}_1, \Omega_x\times C^1((0,T)\times Y))$ to the bilaterally coupled model.
\label{Diff:FullExist}
\end{Theorem}

\begin{proof}
In the setup presented above, Theorem~\ref{TheoremSchulz} ensures the short-time existence of solutions \mbox{$(c,\phi)\in(\mathcal{X}_1)^2$} to the equations
\begin{align*}
    \phi \partial_t c -\nabla \cdot (\hat{\ID}(\phi)\nabla c) &= \hat{\sigma}(\phi) c^2 - \hat{\sigma}(\phi) c,\\ \nonumber
    \partial_t\phi = -\hat{\sigma}(\phi) c,
\end{align*}
with $\hat{\ID},\; \hat{\sigma}$ as constructed in (\ref{EQ:TrafoWRTphi}). Apparently, the second equation corresponds to the level-set function $\Phi$ fulfilling
\begin{align}
\label{local2.5}
    \partial_t \Phi(t,x,y) +c(t,x)|\nabla_y \Phi(t,x,y)| = 0, \quad \forall x\in\Omega,
\end{align}
for $y$ sufficiently close to the zero level-set as required. By the embedding theorems~(\ref{CompactEmbedX1}), $t\mapsto c(t,x)$ is continuous for all $x\in\Omega$. Furthermore, a~solution $\Phi_1$ related to $v_n \equiv 1$ in a~neighborhood of the zero level-set exists and is of class $C^2$ as discussed in Section~\ref{SEC:FlowInducedDiffeomorphisms}. As such, the transformed solution $\Phi(t,x,y) = \Phi_1(C(t,x),x,y)$ with $C(\cdot,x)$ being a~primitive of $c(\cdot,x)$ solving~(\ref{local2.5}) is of class $C^1((0,T)\times Y)$ for each $x\in\Omega$.
\end{proof}

\begin{Bem}
The range of porosities  $(\phi_\text{min},\phi_\text{max})$ the solutions in the above theorem attain depends on the radius $r$ of the tubular neighborhood of the master unit-cell's fluid-solid interface at $s=0$, cf.~Theorem~\ref{TheoremTubNeigh}. More precisely, $\phi_\text{max}-\phi_\text{min}$ is bounded by the volume of the tubular neighborhood. In the case of circular geometries as illustrated in Example~\ref{Ex:DiffeoCircle}, we can allow for $(\phi_\text{min},\phi_\text{max})\subset\subset (1-\frac{\pi}{4},1)$, i.e.\ the full range between complete dissolution and clogging. Moreover, under strong assumptions on the smoothness of parameters and additional compatibility conditions of the initial data, Corollary~1 of~\cite{Schulz2017} proves the maximal existence interval of solutions $(c,\phi)$ to end exactly when $\phi$ leaves its admissible range. Given the required regularity, the analogous result holds for $(c,\Phi)$ as the solution of the level-set equation covers the range $(\phi_\text{min},\phi_\text{max})$ by assumption. 
\end{Bem}

\begin{Bem}
Note that a proper level-set function uniquely determines the geometry state of the system but not vice versa. In that sense, the problem of Theorem~\ref{Diff:FullExist} features multiple solutions $(c,\Phi)$. However, all these solutions are equivalent as they describe the same geometry evolution. By Theorem 4.2 in~\cite{Schulz2017}, the solution $(c,\phi)\in\mathcal{X}_1^2$ is unique. By the one-on-one relation between geometry state and porosity, this translates to the uniqueness of the solid-part within each unit-cell $Y(t,x)$. Fixing an initial value $\Phi_0$, also the solution $(c,\Phi)$ is uniquely determined. 
\end{Bem}

\section{Smooth parameter dependence and existence for diffusive-advective transport}
\label{SEC:AdvTransport}
In this section, we consider an extension of the model analyzed in Section~\ref{SEC:DiffTransport} involving fluid flow and including advective solute transport. To do so, we follow a~similar strategy as in Section~\ref{SEC:DiffTransport}. After fixing the setting in Section~\ref{SEC:SettingAdvect}, we prove smooth dependence of the permeability tensor on the geometry using a~variant of the implicit function theorem in Section~\ref{SEC:perm}. Building upon those results the continuous dependence of Darcy velocity and pressure on the permeability field is shown. Finally, we obtain local-in-time existence results for the diffusive-advective transport case with partial one-way micro-to-macro coupling in Section~\ref{SEC:AdvectExistence}.

\subsection{Setting}
\label{SEC:SettingAdvect}
In the following, we consider solute transport by diffusion and advection:
\begin{align}
\label{EQ:transportAdvect}
    \phi \partial_t c +\nabla\cdot (vc)- \nabla\cdot(\mathbb{D}\nabla c) =  \sigma f(c) - \partial_t \phi c\quad \text{in } (0,T)\times \Omega
\end{align}
which is coupled to Darcy's equation
\begin{align}
\label{EQ:DarcyAdvect}
   v &= -\mathbb{K} \nabla p &&\text{in } \Omega,\; t\in(0,T) ,\\
\nabla \cdot v &=\tilde{f} && \text{in } \Omega,\; t\in(0,T).   \nonumber 
\end{align}
In comparison to the equations stated in Section~\ref{SEC:Model}, we allow for a general constant-in-time divergence $\tilde{f}$. In contrast to the previous analysis, we now consider a~partial coupling from the microscopic to the macroscopic scale as in Section~\ref{SEC:DiffExist1Sided}. That is, we suppose the evolution of the underlying geometry is a-priori given by a~sufficiently smooth path of diffeomorphisms. Finally,  the model regarded in this chapter is closed by the two effective tensors and associated cell problems~(\ref{Def: DiffusionTensor}),~(\ref{Def: PermTensor}).

\subsection{Continuous dependence of permeability tensors}
\label{SEC:perm}
Using a~similar strategy as in Section~\ref{SEC:diff}, we also show the continuous dependence of the permeability tensor on the geometry evolution. 

In order to capture the underlying stationary Stokes equations in~(\ref{Def: PermTensor}), we introduce the following function spaces for the velocity field $v$ and pressure field $q$ for $k\geq 1$, cf.~\cite{Galdi}:
\begin{align*}
    H^k_v(\mathcal{P}) &= \left\lbrace v \in (H_\#^k(\mathcal{P}))^d:\; \text{Tr}_\mathcal{P}(v) = 0 \right\rbrace,\\
    H^{k-1}_q(\mathcal{P}) &= \left\lbrace q \in H_\#^{k-1}(\mathcal{P}):\;\int_\mathcal{P} q \; dy = 0 \right\rbrace, \\
    H^k_{v,\Box}(\mathcal{P}) &= \left\lbrace v\in  H^k_v(\mathcal{P}):\;  \nabla \cdot v = 0 \text{ on } U \right\rbrace
\end{align*}
for some fixed neighborhood $\partial Y \subset U \subset \bar{Y}$, cf.~(\ref{EQ:DiffBoundaryPreserving}). In order to formulate the notion of weak solutions stated in~\cite{Galdi}, we furthermore introduce the space 
\begin{align*}
H^1_{v,\sigma}(\mathcal{P}):=\left\lbrace u\in H^1_{v}(\mathcal{P}): \;\nabla \cdot u =0 \right\rbrace     
\end{align*}
of solenoidal functions in $H^1_{v}(\mathcal{P})$. As in Section~\ref{SEC:diff}, we are required to consider the general inhomogeneous class of Stokes equations.
Therefore, the weak form of the Stokes problem with general force term $f\in (L^2(\mathcal{P}))^d$ and inhomogeneity $j\in H^1(\mathcal{P})\cap L^2_0(\mathcal{P})$ with $j=0$ in a~neighborhood of $\partial Y$ reads: Find $(u,p)\in H^1_v(\mathcal{P}) \times H^0_q(\mathcal{P})$ such that
\begin{align}
\label{EQ: PermWeakForm}
    \int_\mathcal{P} \nabla u : \nabla v \; dy  &= \int_\mathcal{P} f v \; dy, \nonumber \\
    \int_\mathcal{P} \nabla u : \nabla \Psi \; dy - \int_\mathcal{P} p  \nabla \cdot \Psi \; dy &= \int_\mathcal{P} f \Psi \; dy,\\
    \nabla \cdot u &= j, \nonumber
\end{align}
for all $(v,\Psi)\in (H^1_{v,\sigma}(\mathcal{P}) \times C^\infty_0(\mathcal{P}))$.

Following the procedure of Section~\ref{SEC:diff}, we next implement higher regularity of solutions to the Stokes problem given a~sufficiently smooth interior boundary $\partial^\text{int}\mathcal{P}$.

\begin{Lemma} \textit{Regularity Stokes} \\
Let $\partial^\text{int}\mathcal{P}$ be of class $C^2$, $f\in (L^2(\mathcal{P}))^d$ and $j\in H^1(\mathcal{P})\cap L^2_0(\mathcal{P})$ vanishing in a~neighborhood of $\partial Y$. Then there exists a~unique solution $(u,p)$ to (\ref{EQ: PermWeakForm}) in  $H^2_v(\mathcal{P}) \times H^1_q(\mathcal{P})$. Furthermore, the associated solution operator is linear and bounded.
\label{Lemma: Regularity Stokes}
\end{Lemma}

\begin{proof}
We start considering the homogeneous case $j\equiv 0$. By standard Hilbert space arguments there exists a~uniquely determined  function $u\in H^1_{v,\sigma}(\mathcal{P})$
such that
\begin{align*}
     \int_\mathcal{P} \nabla u : \nabla v \; dy  = \int_\mathcal{P} f v \; dy, \quad \forall v \in H^1_{v,\sigma}(\mathcal{P}).
\end{align*}
Due to the smoothness of the domain and source term, there exists a~unique pressure field $p\in L^2(\mathcal{P})$ with vanishing average such that
\begin{align*}
     \int_\mathcal{P} \nabla u : \nabla \Psi \; dy - \int_\mathcal{P} p  \nabla \cdot \Psi \; dy &= \int_\mathcal{P} f \Psi \; dy, \quad \forall \Psi \in (C_0^\infty(\mathcal{P}))^2
\end{align*}
by Lemma IV 1.1 in~\cite{Galdi}. Furthermore, Theorem IV 4.1 in~\cite{Galdi} ensures the interior regularity claimed. As the boundary is of class $C^2$, the regularity can be extended to the full domain by \mbox{Theorem IV 5.1 in \cite{Galdi}}. Next, we consider the problem for general inhomogeneities $j$ of the specified class. According to Theorem 3.4 in~\cite{Bogovski} there exists a~function $\beta \in H^2_0(\mathcal{P})$ satisfying $\nabla \cdot \beta = j$. Solving the homogeneous system for $\tilde{f}=f-\Delta \beta$ and adding $\beta$ to the velocity solution, the full statement is shown. Combining the estimates associated to the previous steps, we conclude boundedness of the solution operator. 
\end{proof}

In comparison to the case of elliptic equations in Section~\ref{SEC:diff} we need to slightly change the setting here in order to accommodate for the additional condition on $\nabla \cdot u$ which is not invariant under pullbacks and cannot be additively compensated for without changing the solution itself. More precisely, we switch to a~setting that does not require surjectivity of operators.
Serving the analogous purpose as $F$ defined in equation (\ref{EQ:F}), let us consider the following mappings:
\begin{align*}
    G & :  H^2_{v,\Box}(\mathcal{P}) \times H^1_{q}(\mathcal{P}) \times (-S,S)  \to L^2(\mathcal{P}) \times H^1_{\#}(\mathcal{P}) ,\\
    & (u,p,s) \mapsto \left(h_s^*\Delta h_s^{*-1} u- h_s^*\nabla h_s^{*-1} p ,\;  h_s^*\nabla \cdot h_s^{*-1} u \right),\\
    \tilde{G} & : (-S,S) \to L^2(\mathcal{P}) \times H^1_{\#}(\mathcal{P}), \\
    & (s) \mapsto (h_s^*e_1 ,\; 0),
\end{align*}
with $h$ being a~$C^1$-path of diffeomorphisms in $\text{Diff}^2_{\Box}(\bar{Y})$ and $h_0=\text{id}_{\bar{Y}}$.
As the interior boundary conditions are invariant under pullbacks and already implemented in the underlying function spaces it is sufficient to only map to the function spaces associated to the bulk data $j,\; f$. 
Differentiability of this mapping is established using the same reasoning as above.
Summing up our considerations, we obtain the following theorem. 

\begin{Theorem} \textit{Smooth dependence of $\IK$}\\
Assume a~$C^{2}$ open set $Y\setminus\bar{\mathcal{P}} \subset Y$ being compactly contained in $Y$ and $(u,p)\in H^2_{v,\Box}(\mathcal{P}) \times H^1_q(\mathcal{P})$ such that $G(u,p,0)= \tilde{G}(0)$, i.e.\ $(u,p)$ is a~solution to problem (\ref{EQ: PermWeakForm}) with force term $f=e_1$ and heterogeneity $j \equiv 0$. Furthermore, let $h$ be an~$m$-times differentiable path in $\text{Diff}_{\Box}^2(\bar{Y})$ with $h_0=\text{id}_{\bar{Y}}$. Then there exist a~neighborhood $V$ of zero and a~function $g:V\to H^2_v(\mathcal{P}) \times H^1_{q}(\mathcal{P})$ $m$-times differentiable in $0\in V$, $g(0)=(u,p)$, such that
\begin{align*}
 G(g(s),s)= \tilde{G}(s), \quad \forall s\in V,    
\end{align*}
i.e.\ $h_s^{*-1}g(s)$ solves (\ref{EQ: PermWeakForm}) on $h(\mathcal{P})$ for $s\in V$.
\label{THEOREM2}
\end{Theorem}

\begin{proof}
By calculations similar to Section~\ref{SEC:diff}, we find $G,\; \tilde{G}$ to be differentiable in $s=0$. Furthermore, $G(\cdot,\cdot,s)$ is a~bounded linear operator for fixed $s$. The unique solvability of 
\begin{align*}
    G(u,p,s)= \tilde{G}(s)
\end{align*}
is guaranteed for all $s\in V$ by Lemma~\ref{Lemma: Regularity Stokes}, i.e.\ we can properly define the map $g(s)$. Finally, we obtain the estimate
\begin{align*}
    ||G(u,p;0)||_{L^2(\mathcal{P}) \times H^1(\mathcal{P})} \geq C^{-1} ||(u,p)||_{H^2(\mathcal{P}) \times H^1(\mathcal{P})} \quad \forall (u,p) \in H^2_{v,\Box}(\mathcal{P}) \times H^1_q(\mathcal{P})
\end{align*}
using the operator norm $C<\infty$ of the Stokes solution operator on $\mathcal{P}$ introduced in Lemma~\ref{Lemma: Regularity Stokes}. The assertion is a~consequence of Theorem~\ref{TheoremSimon}. This is due to the fact that all values in the second image space of $G$ vanish close to $\partial Y$ by the restrictions on the preimage-spaces. 
\end{proof}

This result immediately translate to the permeability tensor: 
\begin{Korollar}
\label{Cor:Theorem2}
Under the assumptions of Theorem \ref{THEOREM2} the mapping 
\begin{align*}
    R: V \to \mathbb{R} ,\quad  s\mapsto \int_{h_s(\mathcal{P})}  h_s^{*-1}g_u(s) \; dy
\end{align*} 
is $m$-times differentiable. Therefore, the permeability tensor depends differentiably on variations of $s$.
\end{Korollar}

\subsection{Continuous dependence of Darcy velocity and pressure}

The Darcy velocity field $v$ enters the first-order term of the transport equation~(\ref{EQ:transport}) as a~parameter. Again, in order to apply linear parabolic theory as in Section \ref{SEC:DiffExist1Sided}, we first need to establish sufficient regularity of $v$ which immediately rises the question of dependence on the permeability $\IK$. A~special difficulty arises from the fact that the stationary Darcy equation~(\ref{EQ:Darcy}) acts on time-slices of the space-time cylinder $\Omega_T$.
In this chapter we investigate the effect of permeability tensors continuously depending on space and time on the flow field fulfilling Darcy's equation. The following result establishes local Lipschitz continuity with respect to the $L^2$-norm in the velocity field. Assume a~typical flow-channel scenario with flux boundary conditions on $\partial\Omega_\text{flux}$ and an~outlet $\partial\Omega_\text{Dir}$ with zero Dirichlet data for the pressure with $\partial\Omega = \partial\Omega_\text{Dir} \cup \partial\Omega_\text{flux}$ and Hausdorff measure $\mathcal{H}^{d-1} (\partial\Omega_\text{Dir}) >0$. 
First, we consider (\ref{EQ:Darcy}) as an~elliptic equation for the pressure $p$ in the following weak form with a~general source term $\tilde{f}\in L^2(\Omega)$:

Find $p\in H_{\text{Dir}}^1(\Omega):=\left\lbrace v\in H^1(\Omega):\; v=0 \text{ on } \Omega_{\text{Dir}}\right\rbrace$ such that for all \mbox{$q \in H_{\text{Dir}}^1(\Omega)$}
\begin{align}
    \label{EQ:DarcyWeak}
    \int_{\Omega}  \mathbb{K} \nabla p \cdot \nabla q \; dx - (g_{\text{flux}},q)_{L^2(\partial\Omega_\text{flux}) }= \int_{\Omega} \tilde{f} q \; dx. 
\end{align}
The associated velocity field is then given as
\begin{align*}
     v=-\IK \nabla p.
\end{align*}
As such, we can establish regularity of $v$ by analyzing the pressure equation~(\ref{EQ:DarcyWeak}) and deduce the relevant properties from~$p$.
Next, we follow the approach taken in \cite{Celepi06} where continuous dependence on parameters in case of the Brinkman-Forchheimer equation was established. Considering two pressure solutions to Darcy's equation $(p_1,p_2)$ related to a~pair of coefficients $(\mathbb{K}_1, \mathbb{K}_2)$, we test the difference of the weak formulations with $p_2-p_1$. Using suitable estimates on the new right-hand side the assertion is established. More precisely, the following statement holds:

\begin{Lemma}
\label{Lemma:ContiFlow}
Let $\Omega$ be a~Lipschitz bounded and connected domain with flux boundary conditions on $\partial\Omega_\text{flux}$ and homogeneous pressure boundary conditions on $\partial\Omega_\text{Dir}$:
\begin{align*}
    \IK\nabla p_{|\partial\Omega_\text{flux}} \cdot \nu= g_\text{flux} \in L^2(\partial\Omega_\text{flux}), \quad  p = 0 \in H^{\frac{1}{2}} (\partial\Omega_\text{Dir}).
\end{align*}
Assume that the Dirichlet boundary part $\partial\Omega_\text{Dir}$ is of positive measure. Furthermore, let $\mathbb{K}_1, \mathbb{K}_2 \in \left( L^\infty\left({\Omega}\right) \right)^{d \times d}$ be permeability tensor fields uniformly $\lambda >0$ -coercive and uniformly bounded in the Frobenius norm a.e. by $K_{max}$. In addition, let $\tilde{f}\in L^2(\Omega)$. Denoting the respective solutions to Darcy's equation by \mbox{$(v_1,p_1), \; (v_2,p_2) \in L^2(\Omega)\times H^1(\Omega)$} the following estimates hold:
\begin{align*}
     ||p_2-p_1||_{H^1(\Omega)} \leq C ||\mathbb{K}_2-\mathbb{K}_1||_{L^\infty(\Omega)},\\
     ||v_2-v_1||_{L^2(\Omega)} \leq C ||\mathbb{K}_2-\mathbb{K}_1||_{L^\infty(\Omega)}, \nonumber
\end{align*}
where the constant $C$ only depends on $\lambda$, $K_{max}$, $\Omega$ and the Darcy data $g_{\text{flux}},\tilde{f}$. 
\end{Lemma}

\begin{proof}
First of all, we note that standard elliptic theory guarantees the existence of solutions $p_1,p_2\in H_{\text{Dir}}^1(\Omega)$  to~(\ref{EQ:DarcyWeak}).
In a~first step towards continuous dependence of solutions on $\IK$, we show uniform boundedness of $p$ in the $H^1$-seminorm by revisiting well-established energy methods. As $\mathbb{K}$ is symmetric positive definite, there exists a~unique symmetric root $\mathbb{K}^\frac{1}{2}$. As $p$ itself is an~admissible test function, the weak formulation directly yields for an~arbitrary $\epsilon>0$:
\begin{align*}
    \lambda || \nabla p ||^2_{L^2(\Omega)} &\leq || \mathbb{K}^\frac{1}{2} \nabla p ||^2_{L^2(\Omega)} \leq (g_{\text{flux}},p)_{L^2(\partial\Omega_\text{flux}) } +||\tilde{f}||_{L^2(\Omega)} ||p||_{H^1(\Omega)}  \nonumber\\
    &\leq C(\epsilon) \left(||g_\text{flux}||^2_{L^2(\partial\Omega_\text{flux})}+ ||\tilde{f}||^2_{L^2(\Omega)} \right) +\epsilon ||\nabla p||^2_{L^2(\Omega)} 
\end{align*}
using Young's and Poincaré's inequality as well as the trace theorem. The eigenvalues of $\mathbb{K}$ being bounded from below by $\lambda$ and choosing $\epsilon$ small enough, we have
\begin{align*}
    ||\nabla p ||_{L^2(\Omega)}&\leq C \left( || g_\text{flux}||_{L^2(\partial\Omega_\text{flux})}+||\tilde{f}||_{L^2(\Omega)} \right)
\end{align*}
with $C$ depending on Poincaré's constant and $\lambda$ only.

Next, assume $p_1, \; p_2$ to be pressure solutions to the Darcy problems respective to $\IK_1,\IK_2$. Testing the difference of the weak formulations with $p:=p_2-p_1$ leads to
\begin{align*}
    ||\mathbb{K}_1^\frac{1}{2} \nabla p ||^2_{L^2(\Omega)} &= \int_\Omega (\mathbb{K}_1 -\mathbb{K}_2) \nabla p_2 \cdot \nabla p \; dx \\ & \leq \esssup\limits_{x\in \Omega}(||\mathbb{K}_1 -\mathbb{K}_2||_F) \left( C(\epsilon) ||\nabla p_2||^2_{L^2(\Omega)} + \epsilon ||\nabla p||^2_{L^2(\Omega)} \right) \nonumber \\
    & \leq \esssup\limits_{x\in \Omega}(||\mathbb{K}_1 -\mathbb{K}_2||_F)  C(\epsilon) ||\nabla p_2||^2_{L^2(\Omega)} + 2K_{max} \epsilon ||\nabla p||^2_{L^2(\Omega)}, \nonumber
\end{align*}
where $||.||_F$ denotes the Frobenius norm of the matrices. Using the uniform boundedness in the gradient of $p_2$ established in the first step and choice of $\epsilon$ small enough yields the assertion on the pressure. Calculating
\begin{align*}
    ||v_1 -v_2||_{L^2(\Omega)} &= ||\mathbb{K}_1 \nabla p_1 -\mathbb{K}_2 \nabla p_2 ||_{L^2(\Omega)}  \\
    &\leq ||\mathbb{K}_1||_{L^\infty(\Omega)}  ||\nabla p ||_{L^2(\Omega)} + ||\mathbb{K}_1-\mathbb{K}_2||_{L^\infty(\Omega)}  ||\nabla p_2 ||_{L^2(\Omega)} \nonumber\\
    &\leq C ||\mathbb{K}_2-\mathbb{K}_1||_{L^\infty(\Omega)},\nonumber
\end{align*}
we obtain the assertion on the velocities applying the same reasoning.
\end{proof}

Unfortunately, the last result cannot be generalized to stronger norms in $v$ by demanding higher regularity to the data, coefficient or $\Omega$ due to the mixed boundary conditions. For a~counter-example in a~setting of maximal smoothness see \cite{Eliahu}. 

Yet, in the case of pure Dirichlet boundaries, it is straightforward to show that the above estimate indeed can be refined to hold in $L^\infty(\Omega)$ by standard elliptic theory~\cite{Evans}. We capture this observation in the following Corollary.

\begin{Korollar}
\label{Cor:DarcyLinfty}
Let the assumptions of Lemma~\ref{Lemma:ContiFlow} hold with $\partial \Omega_\text{flux}=\emptyset$, a~$C^3$-domain $\Omega\subset\mathbb{R}^d$, $\tilde{f} \in H^1 (\Omega)$ and $d\in\{2,3\}$. Furthermore, let $\IK_1,\IK_2 \in (C^2(\bar{\Omega}))^{d\times d}$ be uniformly coercive and uniformly bounded. Then 
\begin{align*}
     ||v_2-v_1||_{L^\infty(\Omega)} \leq C ||\mathbb{K}_2-\mathbb{K}_1||^{1-\frac{d}{4}}_{L^\infty(\Omega)}.
\end{align*}
\end{Korollar}

\begin{proof}
Using standard elliptic theory, the solutions $p_1,p_2$ to~(\ref{EQ:DarcyAdvect}) are uniformly bounded in $W^{3,2}(\Omega)$, cf.~\cite{Evans}. By the Gagliardo-Nirenberg interpolation theorem~\cite{InterpolationIneq} we have for $p=p_2-p_1$
\begin{align*}
    ||\nabla p||_{L^\infty(\Omega)} \leq C_1 ||D^3p||^\frac{d}{4}_{L^2(\Omega)} ||\nabla p||_{L^2(\Omega)}^{1-\frac{d}{4}} + C_2  ||\nabla p||_{L^2(\Omega)}.
\end{align*}
Using the uniform boundedness of $||D^3p||_{L^2(\Omega)}$ and Lemma~\ref{Lemma:ContiFlow}, we see
\begin{align*}
    ||\nabla p||_{L^\infty(\Omega)} \leq C ||\mathbb{K}_2-\mathbb{K}_1||_{L^\infty(\Omega)}^{1-\frac{d}{4}}
\end{align*}
which translates to the statement claimed.
\end{proof}

\begin{Bem}
\label{Rem:ContinuousV}
In case $\IK:\Omega_T\to\mathbb{R}^{d\times d}$ fulfils the conditions of Corollary~\ref{Cor:DarcyLinfty} for each time-slice and the additional condition $\IK \in C\left(0,T; (L^\infty(\Omega))^{d\times d}\right)$ holds, Darcy's velocity $v:\Omega_T\to\mathbb{R}^d$ is a~continuous function on the whole space-time cylinder.
\end{Bem}

\subsection{Existence for partial coupling}
\label{SEC:AdvectExistence}
Analogously to Section~\ref{SEC:DiffExist1Sided}, we use our prior results to prove local-in-time existence of solutions to the model specified in Section~\ref{SEC:SettingAdvect}. Using the prior regularity results of this section sufficient smoothness of the Darcy velocity field is established. Again, we redefine the functions $\phi,\;  \sigma,\; \ID,\; \IK$ from mappings on $\Omega_T$, cf.~(\ref{EQ:transportDIFF}), to mappings of the real valued order parameter without change of notation. In order to avoid additional difficulties due to degeneracy, we require
\begin{align}
\label{AssumptionsOneSidedAdvect}
   \forall s\in (-S,S): 0< \phi(s) <1,  \quad \sigma(s) >0, \quad \ID(s)>0, \quad \IK(s)>0,
\end{align}
to hold along the prescribed path $h$ of diffeomorphisms within the master unit-cell $Y^*$.

\begin{Theorem} \textit{Existence of strong solutions, partial coupling, advective transport}\\
Let $\Omega\subset\mathbb{R}^d,\; d\in\{2,3\}$ be a~$C^3$-domain. Concerning the transport equation~(\ref{EQ:transportAdvect}), let $r>d+2$ hold with initial conditions $c_0\in W^{2-\frac{2}{r},r}_r(\Omega)$ and boundary conditions $C_0 \in W^{1-\frac{1}{2r}, 2-\frac{1}{r}}(\partial \Omega_T)$ being compatible in the sense of $C_0(0,\cdot)=c_0$ on $\partial\Omega$. Concerning Darcy's equation~(\ref{EQ:DarcyAdvect}), let $p_{|\partial \Omega} = 0$ and $\tilde{f}\in H^2(\Omega)$. 
Furthermore, let the evolution of the pore-space geometry be given by an~order parameter $s\in C^1([0,T_1);C^2(\bar{\Omega}))$, \mbox{$s(t,x)\in (-S+\epsilon,S-\epsilon)$}, $ \epsilon>0\;  \forall (t,x)\in \Omega_T$, and a~path $h \in C^3(-S,S;\text{Diff}_{\Box}^{2,1}(\bar{Y}))$ of diffeomorphisms such that $h_0 = \text{id}_{\bar{Y}}$. Assume the initial inclusion $Y^*\setminus \bar{\mathcal{P}} \subset Y^*$ to be a~compactly contained $C^{2,1}$ open set. Let~(\ref{AssumptionsOneSidedAdvect}) hold true and the reaction rate $f$ be locally Lipschitz. Then there exists a~time $0<T\leq T_1$ such that the system (\ref{EQ:transportAdvect}),~(\ref{EQ:DarcyAdvect}) admits a~unique solution \mbox{$(c,v,p)\in \mathcal{X}_1\times C(0,T,L^{2}(\Omega)) \times C(0,T,W^{1,2}(\Omega))$}.
\label{Theorem:AdvectOneSideExist}
\end{Theorem}

\begin{proof}
The proof of this assertion follows along the lines of the proof of Theorem~\ref{Theorem:DiffOneSideExist}. In addition, we must now also consider the smoothness of $\IK(s(t,x))$ as well as of the associated Darcy velocity field $v(t,x)$.
By the assumptions on the geometry alteration and Corollary~\ref{Cor:Theorem2}, \mbox{$\IK(s(t,\cdot))\in \left(C^2(\bar{\Omega})\right)^{d\times d}$} for every $t\in[0,T_1)$. Due to~(\ref{AssumptionsOneSidedAdvect}) and Remark~\ref{Rem:ContinuousV} we obtain \mbox{$v \in C(\overline{\Omega_T})$} for some $0<T\leq T_1$. By standard Sobolev embedding theorems, also $\nabla \cdot v\in C(\overline{\Omega_T})$ holds true.  
As such, we ensure the regularity of $v$ as required for coefficients in Theorem~\ref{TheoremLadyzenkaya} and may therefore proceed with the fixed point argument analogously to Theorem~\ref{Theorem:DiffOneSideExist}. As mentioned before, the solutions $(v(t,\cdot),p(t,\cdot))$ for Darcy's equation are uniquely determined for each time-slice $t$. Repeating the arguments following~(\ref{EQ:DiffUniqueness}), uniqueness for the full solution triplet is obtained.
\end{proof}

\section{Conclusion}
\label{SEC: Conclusion}
In this research, we presented local-in-time existence results to a common class of multi-scale models for reactive transport in evolving porous media. Describing geometry alterations by diffeomorphisms, smooth dependence of the diffusion and permeability tensors on the evolving domain is proven. Being computed as affine-linear functionals of solutions to elliptic and Stokes-type PDEs, the method of transformation to a reference domain in combination with the implicit function theorem was applied for this purpose. Subsequently, we leveraged the resulting smoothness in the micro-macro coupling to short-time existence results to the partially and fully coupled system in the diffusive transport case. Similar results are presented for the partial coupling in the advective case additionally incorporating Darcy's equation. 

As the second major aspect, we proved the induction of suitable paths of diffeomorphisms by the level-set equation. Therefore, our results cover a~broad range of possible underlying microscopic geometries. In particular, no explicit construction of adequate families of diffeomorphisms is needed using this methodology.  

Further research is needed to derive conditions under which the short-time existence results presented in this paper can be extended to global solutions. This especially concerns the adequate treatment of topological changes in the underlying geometry which is not approachable using diffeomorphisms. As such, the methods used in this research are in particular unable to analyze the behavior of solutions in scenarios involving clogging or the complete dissolution of the porous structure. 

In addition, further work is required to extend our results to local-in-time existence for the fully coupled model including advective transport. More precisely, methods need to be refined in order to compensate for the drop of regularity between the permeability tensor field and the advective velocity field. Finally, future effort is required to generalize our findings to the multi-solute and multi-mineral case, leading to a~system of possibly non-linearly coupled parabolic equations on the macroscopic scale.      

\section*{Acknowledgements}
This research was supported by the DFG Research Training Group 2339 Interfaces, Complex Structures, and Singular Limits. Nadja Ray was also supported by the DFG Research Unit 2170 MadSoil. We further acknowledge the insightful discussions with Helmut Abels, University of Regensburg.

\appendix
\section{}

\begin{Theorem} (Tubular neighborhood theorem, adapted from Theorem 1.5, \cite{Henry})\\
\label{TheoremTubNeigh}
Let $\Omega\subset \mathbb{R}^d$ open have a~$C^{m,\alpha}$-regular boundary, $2\leq m\leq \infty$. There exists $r>0$ so that if
\begin{align*}
    B_r(\partial \Omega) &= \{x:\text{ dist}(x,\partial \Omega)<r\}, \\
    \pi(x) &= \text{the point of } \partial \Omega \text{ nearest to } x, \\
    t(x) &= \pm \text{dist}(x,\partial\Omega) \quad  \text{('+' outside, '-' inside)},
\end{align*}
then $t(\cdot):B_r(\partial\Omega) \to (-r,r), \; \pi(\cdot): B_r(\partial \Omega)\to \partial\Omega$ are well defined, $\pi$ is a~$C^{m-1,\alpha}$-retraction onto $\partial\Omega$ ($\pi(x)=x$ when $x\in\partial\Omega$) and t has the same smoothness as $\partial\Omega$. Further
\begin{align*}
x \mapsto (t(x),\pi(x)): \; B_r (\partial\Omega) \to (-r,r)\times \partial\Omega 
\end{align*}
is a~$C^{m-1,\alpha}$-diffeomorphism with inverse
\begin{align*}
    (t,\zeta) \mapsto \zeta +t\nu(\zeta): \; (-r,r) \times \partial\Omega \to B_r(\partial\Omega).
\end{align*}
$t(\cdot)$ is the unique solution to $|\nabla t(x)|=1$ in $B_r(\partial\Omega)$.
The largest choice of $r$ is $r=1/\max |k|$, where $k$ is the sectional curvature of the boundary in any (tangent) direction at any point of $\partial\Omega$.
\end{Theorem}

\begin{Theorem} (Local-in-time existence of strong solutions, Theorem 4.1, \cite{Schulz2017})\\
\label{TheoremSchulz}
Let the system of PDE's be given by
\begin{align}
\label{TheoremSchulzLocal1}
    \phi \partial_t c - \nabla \cdot (D(\phi)\nabla c) &= \tau(\phi) c^2 - \sigma(\phi) c && \text{in } \Omega_T, \nonumber \\ 
    \partial_t \phi &= -\tau(\phi)c && \text{in } \Omega_T, \nonumber \\ 
    c(t,x) &= 0 && \text{on } \partial \Omega_T, \\
    c(0,x) &= c_0(x) &&\text{in } \Omega, \nonumber \\
    \phi(0,x) &= \phi_0(x) &&\text{in } \Omega, \nonumber 
\end{align}
Let $\Omega \subset \mathbb{R}^d$, $d\in\{2,3\}$, be a~domain with $C^2$-smooth boundary $\partial\Omega$, $r>d+2$,  \mbox{$c_0\in W^{2-\frac{2}{r},r}(\Omega)$}, $c_0 \geq 0$, satisfying the compatibility condition $c_{0|\partial\Omega} \equiv 0$ and let $\phi_0 \in W^{2,r}(\Omega)$ hold with \mbox{$\phi_0(x)\in (\delta, 1-\delta) \subset (0,1)$} for all $x \in \Omega$ and some $\delta \in (0,\frac{1}{2})$. Furthermore, let $D\in C^1((0,1))$ be a~positive scalar function and $\sigma\in C((0,1))$, $\tau\in C^2((0,1))$. Then, there exists a~constant $T>0$ and at least one strong solution $(c,\phi) \in \mathcal{X}_1^2$ solving (\ref{TheoremSchulzLocal1}) with
\begin{align*}
    \mathcal{X}_1:=W^{1,2}_r(\Omega_T) = L^r(0,T;W^{2,r}(\Omega)) \cap W^{1,r}(0,T;L^r(\Omega)).
\end{align*}
\end{Theorem}

\begin{Theorem} (Parabolic regularity, specialized form of Theorem 9.1, Chapter IV \cite{Ladyzhenskaya})\\
\label{TheoremLadyzenkaya}
Let a~parabolic problem be given by
\begin{align}
\label{EQ:AppendixParabol}
    \mathcal{L}\left(t,x,\frac{\partial}{\partial t},\frac{\partial}{\partial x}\right) u(t,x) &=f(t,x) && \text{in } \Omega_T, \nonumber \\
    u&=U_0 && \text{on } \partial \Omega_T, \\
    u(0,\cdot) &= u_0 && \text{in } \Omega, \nonumber
\end{align}
with the uniformly parabolic operator $\mathcal{L}$ in non-divergence form 
\begin{align*}
    \mathcal{L}\left(t,x,\frac{\partial}{\partial t}, \frac{\partial}{\partial x}\right) =\frac{\partial u}{\partial t}-\sum\limits_{i,j=1}^d a_{i,j}(x,t)\frac{\partial^2u}{\partial x_i \partial x_j} + \sum\limits_{i=1}^d a_i(x,t)\frac{\partial u}{\partial x_i}+a(x,t) u.
\end{align*}
Let $r>d+2$. Suppose that the coefficients $a_{i,j}$ of the operator $\mathcal{L}$ are bounded and continuous in $\Omega_T$, while the coefficients $a_i$ and $a$ have finite norms $||a_i||_{L^r(\Omega_T)}$ and $||a||_{L^r(\Omega_T)}$. Furthermore, let $\Omega$ be a~bounded domain of class $C^2$. Then for any $f\in L^r(\Omega_T)$, Dirichlet data \mbox{$U_0 \in W^{1-\frac{1}{2r}, 2-\frac{1}{r}}_r(\partial \Omega_T)$} and $u_0 \in W^{2-\frac{2}{r},r}(\Omega)$ being compatible in the sense of $U_0(0,\cdot)=u_0$ on $\partial\Omega$, problem (\ref{EQ:AppendixParabol}) has a~unique solution $u\in W^{1,2}_r(\Omega_T)$ satisfying the a-priori estimate

\begin{align*}
    ||u||_{W^{1,2}_r(\Omega_T)}\leq C_p \left(||u_0||_{W^{2-\frac{2}{r},r}(\Omega)} + ||U_0||_{W^{1-\frac{1}{2r}, 2-\frac{1}{r}}_r(\partial \Omega_T)}+ ||f||_{L^r(\Omega_T)}\right).
\end{align*}
\end{Theorem}

\begin{Theorem}
\label{TheoremSimon} (Differentiability of an~implicit equation solution, Theorem 6 \cite{Simon}) \\
We give us
\begin{itemize}
\item an~open set $\mathcal{U}$ in a~Banach space $U$, $u_0 \in \mathcal{U}$, two reflexive Banach spaces $A$ and $B$,
\item a~map $F:\mathcal{U}\times A~\to B$, such that $F(u;\cdot)\in \mathcal{L}(A;B)$ for all $u\in\mathcal{U}$,
\item a~function $m:\mathcal{U}\to A$, and a~function $f:\mathcal{U}\to B$, such that
\begin{align*}
    F(u,m(u))=f(u) \quad \forall u\in\mathcal{U}.
\end{align*}
\end{itemize}
\begin{enumerate}
    \item Assume that $u \mapsto F(u;\cdot)$ is differentiable at $u_0$ into $\mathcal{L}(A;B)$, $f$ is differentiable at $u_0$,
    \begin{align*}
        ||F(u_0;x)||_B \geq \alpha||x||_A \quad \forall x\in A, \quad \quad \text{for some } \alpha>0.
    \end{align*}
    Then the map $u\mapsto m(u)$ is differentiable at $u_0$. It's derivative $m^\prime(u_0;\cdot)$ is the unique solution of
    \begin{align*}
        F(u_0;m^\prime (u_0;v)) = f^\prime(u_0;v)-\partial_u F(u_0;m(u_0);v)\quad \forall v \in U.
    \end{align*}
    \item In addition, assume that for some integer $k\geq 1$,
    \begin{align*}
        u \mapsto F(u;\cdot) \text{ and } f \text{ are } k \text{ times differentiable at } u_0.
    \end{align*}
    Then, the map $u\mapsto m(u)$ is $k$ times differentiable at $u_0$.
\end{enumerate}

\end{Theorem}
%\bibliographystyle{plainurl}
%\bibliography{literature}
\printbibliography[heading=bibintoc]
\end{document}